\documentclass{article}

\usepackage{amsmath}
\usepackage{amssymb}
\usepackage[pdftex,bookmarksnumbered=true]{hyperref}
\usepackage{tikz}
\usepackage[a4paper,margin=3cm]{geometry}

\title{Model completion of scaled lattices and co-Heyting
algebras of $p$\--adic semi-algebraic sets%
\thanks{%
  Keywords: model-theory, p-adic, scaled lattice, Heyting algebra,
  quantifier elimination, decidability, model-completion, uniform
  interpolant.\hfill\break 
  MSC classes: 03C10, 06D20, 06D99.}}

\author{Luck Darni\`ere}

\newtheorem{lemma}[subsection]{Lemma}
\newtheorem{proposition}[subsection]{Proposition}
\newtheorem{theorem}[subsection]{Theorem}
\newtheorem{corollary}[subsection]{Corollary}
\newtheorem{fact}[subsection]{Fact}

\newtheorem{theoremAlph}{Theorem}

\newenvironment{remark}%
   {\refstepcounter{subsection}%
        \medbreak\noindent{\bf Remark \thesubsection\space}}%
   {\par\medbreak}%
\newenvironment{example}%
   {\refstepcounter{subsection}%
        \medbreak\noindent{\bf Example \thesubsection\space}}%
   {\par\medbreak}%

\newenvironment{proof}%
   {\medbreak\noindent{\it Proof:\space}}%
   {\par\noindent\vrule height 5pt width 5pt depth 0pt\smallbreak}%

\newcommand{\df}{\bf}
\renewcommand{\cal}{\mathcal}
\let\sauvegardetiret=\-
\renewcommand{\-}[1]{\ifx#1-\penalty10000\hbox{-\relax}\penalty10000\else\sauvegardetiret#1\fi}
\newcommand{\Sp}{{\rm Spec}}
\newcommand{\tq}{\,|\,}
\newcommand{\nin}{\not\in}
\newcommand{\vect}{\overrightarrow}
\newcommand{\Lssi}{\;\Longleftrightarrow\;}   % Ssi long
\newcommand{\donc}{\Rightarrow}          % Implication courte
\newcommand{\Ldonc}{\;\Longrightarrow\;} % Implication longue
\newcommand{\NN}{{\mathbf N}}
\newcommand{\ZZ}{{\mathbf Z}}
\newcommand{\QQ}{{\mathbf Q}}
\newcommand{\RR}{{\mathbf R}}

\newcommand{\p}{{\rm\bf p}}
\newcommand{\cC}{{\cal C}}
\newcommand{\cD}{{\cal D}}
\newcommand{\cI}{{\cal I}}
\newcommand{\cK}{{\cal K}}
\newcommand{\cL}{{\cal L}}
\newcommand{\cN}{{\cal N}}
\newcommand{\cP}{{\cal P}}
\newcommand{\cU}{{\cal U}}
\newcommand{\cV}{{\cal V}}
\newcommand{\cY}{{\cal Y}}
\newcommand{\UN}{{\mathbf 1}}
\newcommand{\ZERO}{{\mathbf 0}}
\newcommand{\meet}{\land}
\newcommand{\join}{\lor}
\newcommand{\mmeet}{\mathop{\meet\mskip-6mu\relax\meet}}
\newcommand{\jjoin}{\mathop{\join\mskip-6mu\relax\join}}
\newcommand{\conj}{\bigwedge}
\newcommand{\disj}{\bigvee}
\newcommand{\cconj}{\mathop{\bigwedge\mskip-9mu\relax\bigwedge}}
\newcommand{\ddisj}{\mathop{\bigvee\mskip-9mu\relax\bigvee}}

\DeclareMathOperator{\lzar}{L_{Zar}}
\DeclareMathOperator{\llin}{L_{lin}}
\DeclareMathOperator{\ldef}{L_{def}}
\DeclareMathOperator{\lalin}{L_{lin}^{At}}

\DeclareMathOperator{\card}{Card}
\newcommand{\gen}[1]{\langle{#1}\rangle}
\newcommand{\TC}{\mathrm{TC}}
\newcommand{\SC}{\mathrm{SC}}

\DeclareMathOperator{\pc}{C}

\DeclareMathOperator{\szar}{SC_{Zar}}
\DeclareMathOperator{\slin}{SC_{lin}}
\DeclareMathOperator{\sdef}{SC_{def}}
\DeclareMathOperator{\sazar}{ASC_{Zar}}
\DeclareMathOperator{\salin}{ASC_{lin}}
\DeclareMathOperator{\sadef}{ASC_{def}}

\DeclareMathOperator{\AT}{At}
\DeclareMathOperator{\asc}{asc}
\DeclareMathOperator{\scdim}{sc-dim}

\newcommand{\llat}{{\cal L}_\mathrm{lat}}
\newcommand{\ltc}{{\cal L}_\mathrm{TC}}
\newcommand{\lsc}{{\cal L}_\mathrm{SC}}
\newcommand{\lasc}{{\cal L}_\mathrm{ASC}}

\let\olddownarrow=\downarrow
\renewcommand{\downarrow}{{\olddownarrow}}
\let\olduparrow=\uparrow
\renewcommand{\uparrow}{{\olduparrow}}

\begin{document}

\maketitle

\begin{abstract}
  Let $p$ be prime number, $K$ be a $p$\--adically closed field, $X\subseteq
  K^m$ a semi-algebraic set defined over $K$ and $L(X)$ the lattice of
  semi-algebraic subsets of $X$ which are closed in $X$. We prove that
  the complete theory of $L(X)$ eliminates quantifiers in a
  certain language $\lasc$, the $\lasc$\--structure on $L(X)$ being an
  extension by definition of the lattice structure. Moreover it is decidable,
  contrary to what happens over a real closed field for $m>1$. We classify these
  $\lasc$\--structures up to elementary equivalence, and get in
  particular that the complete theory of $L(K^m)$ only depends on $m$,
  not on $K$ nor even on $p$. As an application we obtain a
  classification of semi-algebraic sets over countable $p$\--adically
  closed fields up to so-called ``pre-algebraic'' homeomorphisms. 
\end{abstract}

\section{Introduction}
\label{se:introduction}

This paper
%\footnote{The present paper is a major revision of
%  \cite{darn-2006}, with extended results.} 
explores the model-theory of various classes of lattices coming from
algebraic geometry, real geometry or $p$\--adic geometry, with special
emphasis on the $p$\--adic case. We obtain model-completion and
decidability results for some of them. Before entering in technical
details let us present the main motivations for this, coming from
geometry and model-theory of course, but also from proof theory and
non-classical logics. 

Given an expansion of a topological field $K$ and a
definable\footnote{We assume the reader to be familiar with basic
  notions from model theory, in particular definable sets and
  functions. In simplest cases ``definable'' boils down to
``semi-algebraic'' over the field $\RR$ of real numbers, or the field
$\QQ_p$ of $p$\--adic numbers.} set $X\subseteq K^n$, we consider the lattice
$L(X)$ of all definable subsets of $X$ which are closed in $X$, and the
ring $\cC(X)$ of all continuous definable functions from $X$ to $K$.
These rings are central objects nowadays in functional analysis,
topology and geometry. To name an example, they are rings of sections
for the sheaf of continuous (say, real valued) functions on a
topological space and as such play the algebraic part in the study of
topological (Hausdorff) spaces. In most cases $L(X)$ is interpretable
in $\cC(X)$, and the prime filter spectrum of $L(X)$ is homeomorphic
to the prime ideal spectrum of $\cC(X)$. Thus $L(X)$ is a first-order
structure interpretable in $\cC(X)$, which captures all the
topological (hence second-order) information on the spectrum of the
ring $\cC(X)$. For the real field $\RR$ for example, it is known since
\cite{grze-1951} that $L(\RR^n)$ is undecidable for every $n\geq2$, hence
can be held liable for the undecidability of $\cC(\RR^n)$. On the
contrary $L(\RR)$ is decidable, and so is the lattice of all closed
subsets of the real line \cite{rabi-1969}. Recently this has been
strengthen and widely generalised in \cite{tres-2017}. However these
undecidability results for $L(X)$ strongly depend on the existence of
irreducible or connected components, hence do not apply to the
$p$\--adic case. But even in that case it is proved in
\cite{darn-tres-2018-tmp} that $\cC(\QQ_p^n)$ is undecidable, this time
for every $n\geq1$. On the contrary, our main result implies that for
every $n\geq1$:
\begin{itemize}
  \item
    $L(\QQ_p^n)$ is decidable, and;
  \item 
    The theory of $L(\QQ_p^n)$ eliminates the quantifier in a natural
    expansion by definition of the lattice language.
\end{itemize}

In another direction, the model-theory of these geometric lattices
$L(X)$ is tightly connected to the existence of uniform interpolants
for propositional calculus in certain intermediate\footnote{An
  intermediate logic is logic which stands between classical and
intuitionist logic.} and modal logics. Indeed, thanks to the
one-to-one correspondence between intermediate logics and varieties of
Heyting algebras, the existence of uniform interpolants for a logic
$\cL$ can be rephrased, {\it mutatis mutandis}, as the existence of a
model-completion for the theory of the corresponding variety $\cV(\cL)$
(see \cite{ghil-zawa-1997}). As lattices of closed sets, all our
lattices $L(X)$ are co-Heyting algebras, that is Heyting algebras with
the order reversed. Moreover, their structure is mostly determined by
the {\em geometry} of $X$. This geometric intuition coming from $X$
was essential to our model-completion results, with natural
axiomatisations, for certain theories of (expansions of) co-Heyting
algebras (theorems~\ref{th:intro-d-sub-mod-comp} and
\ref{th:intro-LQp-decid} below). See also 
\cite{cara-ghil-2017}, \cite{darn-junk-2018} for related results. 
\\

Now we are going to present our results in more detail. They are based
on a careful study of certain expansions of lattices, all inspired by
the  geometric examples of lattices of closed sets over an
$o$\--minimal, $P$\--minimal or $C$\--minimal expansion of a field
$K$. More general structures will be considered in the appendix
section~\ref{se:appendix}, where precise definitions are given of what
we call here ``tame'' topological structures. The point is that there
a good dimension theory for definable sets over such structures. 

\begin{example}\label{ex:Ldef}
  Let $\cK=(K,\dots)$ be a tame topological structure.
  For every definable sets $A,B\subseteq X\subseteq K^m$ let $A- B=\overline{A\setminus B}\cap X$
  where the overline stands for the topological closure. For every $a\in
  A$, the {\df local dimension} of $A$ at $a$ is the maximum of the
  dimensions of the definable neighborhoods of $a$ in $A$, and $A$ is
  called {\df pure dimensional} if it has the same local dimension at
  every point. For every non-negative integer $i$ let
  \begin{displaymath}
    \pc^i(A)=\overline{\{a\in A\tq \dim(A,a)=i\}}\cap A.
  \end{displaymath}
  This is a definable subset of $A$, closed in $A$, called the {\df
  $i$\--pure component} of $A$. We call $\ldef(X)$ the lattice of all
  the definable subsets of $X$ which are closed in $X$, enriched with
  the above functions ``$-$'' and $\pc^i$ for every $i$.
  This is a typical example (Proposition~\ref{pr:geom-scaled}) of what
  we are going to call a $d$\--scaled lattice.
\end{example}

Let $\llat=\{\ZERO,\UN,\join,\meet\}$ be the language of lattices, and
$\lsc=\llat\cup\{-,(C^i)_{i\geq0}\}$ be its expansion by the above function
symbols. Finally let $\sdef(\cK,d)$ be the class of the
$\lsc$\--structures $\ldef(X)$ of Example~\ref{ex:Ldef}, for all the
sets $X$ of dimension at most $d$ definable over $\cK$. A similar
construction can be done over a pure field $K$, with the Zariski
topology on $K^m$. We let $\szar(K,d)$ denote the corresponding class
of $\lsc$\--structures. Surprisingly enough, we prove that in most
cases the universal theory of $\sdef(\cK,d)$ (resp. $\szar(K,d)$) does
not depend on $\cK$ (resp. $K$)!

\begin{theoremAlph}\label{th:intro-univ-theory-Td}
  Given any non-negative integer $d$, the universal theories of
  $\sdef(\cK,d)$ (resp. $\szar(K,d)$) in the language $\lsc$ are
  \emph{the same} for every tame expansion $\cK$ of a topological
  field $K$  (resp. for every infinite field $K$).
\end{theoremAlph}

In order to prove this we give in
Section~\ref{se:notation-definitions} an explicit list of universal
axioms for a theory $T_d$ in $\lsc$, the models of which we call
$d$\--subscaled lattices. All the examples given above are
$d$\--scaled lattices, a natural subclass of $d$\--subscaled lattices
(the class of $d$\--scaled lattices is elementary but not universal).
After some technical preliminaries in
Section~\ref{se:embedd-subsc-latt} we prove in
Section~\ref{se:local-finiteness} that every finitely generated
$d$\--subscaled lattice is finite. Combining this with a linear
representation for finite $d$\--subscaled lattices
(Proposition~\ref{pr:linear-repres}) and with the model-theoretic
compactness theorem, we then prove in Section~\ref{se:lin-rep} that,
whatever is $K$ or $\cK$ in Example~\ref{ex:Ldef}, the
theory of $d$\--subscaled lattices is exactly the universal theory of
$\sdef(\cK,d)$ and of $\szar(K,d)$ (Theorem~\ref{pr:linear-repres}).

A detailed study of the minimal finitely generated extensions of
finite $d$\--subscaled lattices, achieved in
Section~\ref{se:minimal-extensions}, leads us in
Section~\ref{se:model-completion} to the next result
(Theorem~\ref{th:model-completion} and
Corollary~\ref{cor:completions}).

\begin{theoremAlph}\label{th:intro-d-sub-mod-comp}
  For every non-negative integer $d$, the theory of $d$\--subscaled
  lattices admits a model-completion $\bar T_d$ which is finitely
  axiomatizable and $\aleph_0$\--categorical. Moreover, $\bar T_d$ has
  finitely many prime models, hence it is decidable as well as all its
  completions.
\end{theoremAlph}

% \begin{remark}
%   Our $d$\--subscaled lattices form (after language reduction) an
%   elementary class of co-Heyting algebras, closely related to the
%   ``slices'' studied in \cite{hoso-1967}, \cite{ono-1970} and others
%   (see \cite{darn-junk-2011} for a more precise account of this
%   relation). So this paper might also be considered as a contribution
%   to the model-theory of co-Heyting algebras with at most $d$ slices.
% \end{remark}

The axiomatization of $\bar T_d$ given in
Section~\ref{se:model-completion} consists of a pair of axioms
expressing a ``Catenarity'' and a ``Splitting'' property which both
have a natural topological and geometric meaning. In particular the
Splitting Property expresses a very strong form of disconnectedness,
which implies that the models of $\bar T_d$ are atomless. 

\begin{remark}
  Since $0$\--subscaled lattices are exactly non-trivial boolean
  algebras, the above model-completion result for subscaled lattices is
  a generalisation to arbitrary finite dimension $d$ of the classical
  theorem on the model-completion of boolean algebras.
\end{remark}

We develop in Sections~\ref{se:atom-scal-latt} and
\ref{se:mod-comp-ASC} a variant of this quantifier elimination result
in a language $\lasc=\lsc\cup\{\AT_k\}_{k\geq1}$, where each $\AT_k$ is a
unary predicate symbol, to be interpreted as the set of elements which
are the join of exactly $k$ atoms. The model-completion $\bar
T_d^{At}$ that we obtain is axiomatized by the Catenarity Property and
a small restriction of the Splitting Property which preserves the
atoms. This theory $\bar T_d^{At}$ has $\aleph_0$ prime models which can
easily be classified in terms of the prime models of $\bar T_d$, from
which it follows that it is decidable as well as all its completions
(Theorem~\ref{th:model-completion-ASC}).
\\

In the initial version of this paper \cite{darn-2006} we conjectured
that $\ldef(\QQ_p^d)$ might be a natural model of $\bar T_d^{At}$. This
intuition proved to be crucial in the proof of the triangulation of
semi-algebraic sets over a $p$\--adically closed field
\cite{darn-2017b}. Conversely it follows from this triangulation that
$\ldef(X)$ is indeed a model of $\bar T_d^{At}$, for every
semi-algebraic\footnote{A generalization to definable sets over more
general $P$\--minimal fields, if possible, has still to be done.} set
$X\subseteq K^m$ of dimension $\leq d$, from which we derive the following result 
in the last section (Theorem~\ref{th:Ldef-p-adic-super}).

\begin{theoremAlph}\label{th:intro-LQp-decid}
  Let $K$ be a $p$\--adically closed field, $X\subseteq K^m$ a semi-algebraic
  set. Then the complete theory of $L(X)$ is decidable, and eliminates
  quantifier in $\lasc$.
\end{theoremAlph}

The prime $\lasc$\--substructure of $\ldef(X)$ (which is generated by
the empty set) is finite. By Theorem~\ref{th:intro-LQp-decid} it
determines the complete theory of $\ldef(X)$. We expect this invariant
to play also a decisive role in the classification of semi-algebraic
sets over $p$\--adically closed fields up to semi-algebraic
homeomorphisms. Such a classification is far from being achieved, but
a weaker classification, up to ``pre-algebraic'' homeomorphisms over
countable $p$\--adically closed fields, is done here by means of this
invariant (Theorem~\ref{th:pre-alg-iso}).

\section{Notation and definitions}
\label{se:notation-definitions}

$\NN$ denotes the set of non-negative integers, and $\NN^*=\NN\setminus\{0\}$. If
$\cN$ is an unbounded non-empty subset of $\NN$ (resp. the empty subset)
we set $\max\cN=+\infty$ (resp. $\max\cN=-\infty$). The symbols $\subseteq$ and $\subset$
denote respectively inclusion and strict inclusion. The logical
connectives `or', `and' and their iterated forms will be denoted by
$\disj$, $\conj$, $\ddisj$ and $\cconj$ respectively.

\subsection{Lattices and dimension}

In this paper a {\df lattice} is a partially ordered set in which
every finite subset has a greatest lower element and a least greater
bound. This applies in particular to the empty subset, hence our
lattices must have a least and a greatest bound. We let
$\llat=\{\ZERO,\UN,\join,\meet\}$ is the language of lattices, each symbol
having its obvious meaning. 
As usual $b\leq a$ is an abbreviation for $a\join b =a$ and similarly for
$b<a$, $b\geq a$ and $b>a$. Iterated $\join$ and $\meet$ operations are denoted
by $\jjoin_{i\in I}a_i$ and $\mmeet_{i\in I}a_i$ respectively. If the index set
$I$ is empty then $ \jjoin_{i\in I}a_i=\ZERO$ and $\mmeet_{i\in
I}a_i=\UN$. Given a subset $S$ of a lattice $L$, the {\df upper
semi-lattice} generated in $L$ by $S$ is the set of finite joins of
elements of $S$. 

The {\df spectrum} of a lattice $L$ is the set $\Sp(L)$ of all prime
filters of $L$, endowed with the so-called Zariski topology, defined
by taking as a basis of closed sets all the sets
\[
P(a)=\{\p\in\Sp(L)\tq a\in\p\}
\]
for $a$ ranging over $L$. Stone's duality asserts that if
$L$ is distributive (which is always the case in this paper) the map
$a\mapsto P(a)$ is an isomorphism between $L$ and the lattice of closed
subsets $S$ of $\Sp(L)$ such that the complement of $S$ in $\Sp(L)$ is
compact. 

We call a lattice {\df noetherian} if it is isomorphic to the lattice
of closed sets of a noetherian topological space. By Stone's
duality a lattice $L$ is noetherian if and only if its spectrum is a
noetherian topological space. In such a lattice every filter is
principal and every element $a$ can be written uniquely as the join of its
(finitely many) {\df $\join$\--irreducible components}, which are the
maximal elements in the set of $\join$\--irreducible\footnote{An element
  $x$ of a lattice $L$ is $\join$\--irreducible if it is {\em non-zero} and if
$a\join b=x$ implies $a=x$ or $b=x$.} elements of
$L$ smaller than $a$. We denote by $\cI(L)$ the set of all 
$\join$\--irreducible elements of $L$.
\\

We define the {\df dimension of an element $a$ in a lattice $L$},
denoted $\dim_L a$, as the least upper bound (in $\NN\cup\{-\infty,+\infty\}$) of the
set of non-negative integers $n$ such that
\[
  \exists\p_0\subset\cdots\subset\p_n\in P(a). 
\]
This is nothing but the ordinary topological or Krull dimension (defined by
chains of irreducible closed subsets) of the spectral space $P(a)$. By
construction $\dim_L\ZERO=-\infty$ and $\dim_L a\join b=\max(\dim_L a,\dim_L
b)$. The subscript $L$ is necessary since $\dim_L a$ is not preserved by
$\llat$\--embeddings, but we will omit it whenever the ambient lattice
is clear from the context. We let the {\df dimension of $L$} be the
dimension of $\UN_L$ in $L$.

The following definable relation will give us a
first-order definition of the dimension of the elements of $L$ inside
$L$, when $L$ is a co-Heyting algebra (see Fact~\ref{fa:duality}
below):
\[
  b\ll a \iff \forall c\;(c< a\donc b\join c < a)
\]
This is a strict order on $L\setminus\{\ZERO\}$ (but not on $L$ because
$\ZERO\ll a$ for every $a$, including $a=\ZERO$).

\subsection{Co-Heyting algebras}
We let $\ltc=\llat\cup\{-\}$ with `$-$' a binary function symbol. A
$\ltc$\--structure $L$ is a {\df co-Heyting algebra} if its
$\llat$\--reduct is a lattice and if every element $b$ has in $L$ a
{\df topological complement} relatively to every element, denoted
$a-b$. By definition $a-b$ is the least element $c$ such that $a\leq b\join
c$. Equivalently $P(a-b)$ is the topological closure of the relative
complement $P(a)\setminus P(b)$, hence the notation $a-b$. Reversing the order
of a co-Heyting algebra $L$ gives a Heyting algebra $L^*$,
with $b\to a$ in $L^*$ corresponding to $a-b$ in $L$, and every
co-Heyting algebra is of this form. From the theory of Heyting
algebras (see for example \cite{John-1982}) we know that every
co-Heyting algebra is distributive and that the class of all
co-Heyting algebras is a variety (in the sense of universal algebra).
Observe that in co-Heyting algebras the $\ll$ relation is
quantifier-free definable since
\[
b\ll a \iff b\leq a-b.
\]
So it will be preserved by $\ltc$\--embeddings. On the other hand,
dimension will not be preserved in general by $\ltc$\--embeddings.

We will use the following rules, the proof of which are
elementary exercises (using either Stone's duality or
corresponding properties of Heyting algebras). 
\newcommand{\tcref}[1]{$\rm TC_{\ref{#1}}$}
\begin{list}{$\rm\bf TC_\theenumi$:}{\usecounter{enumi}}
\item\label{TC:a=(a-inter-b)-union-(a-b)}
$a=(a\meet b)\join(a-b)$.\\
In particular if $a$ is $\join$\--irreducible then 
$b< a\Ldonc b\ll a$.
\item\label{TC:(x-union-y)-z}
$(a_1\join a_2)-b = (a_1-b)\join (a_2-b)$.
\item\label{TC:(x-y)-y}
$(a-b)-b=a-b$.\\
In particular $(a-b)\meet b\ll a-b \leq a$.
\item\label{TC:x-(y-union-z)}
More generally $a-(b_1\join b_2)=(a-b_1)-b_2$.\\
So if $a-b_1=a$ then $a-(b_1\join b_2)=a-b_2$.
\end{list}

\begin{fact}[Theorem~3.8 in \cite{darn-junk-2011}]
\label{fa:ldim-et-TCdim}
For every element $a\neq\ZERO$ in a co-Heyting algebra $L$, $\dim_L a$ is the
least upper bound of the set of positive integers $n$ such that there
exists $a_0,\dots,a_n\in L$ such that 
\[
\ZERO\neq a_0\ll a_1\ll\cdots\ll a_n\leq a.
\]
\end{fact}

In all the geometric examples given in the introduction, a set $A$ is
said to be pure dimensional if and only if $\dim U=\dim A$ for every
non-empty definable subset $U$ of $A$ which is open in $A$. This
motivates the next definition: given an integer $k$ we say that an
element $a$ of a distributive lattice $L$ is {\df $k$\--pure in $L$}
if and only if
\[
\forall b\in L\ (a-b\neq\ZERO\donc \dim_L a-b =k).
\]
Then either $a=\ZERO$ or $\dim_L a = k$. In the latter case 
we say that $a$ {\df has pure dimension $k$ in $L$}. 
If $L$ is any of the lattices
$\ldef(X)$ or $\lzar(X)$ in Example~\ref{ex:Ldef}, for every $A\in L$ we
will show in Section~\ref{se:appendix} that $\dim_L A$ is exactly the
usual (geometric) dimension of $A$. It follows that $A$ is pure
dimensional in $L$ if and only if it so in the geometric
sense. 
\\

There is a well-established duality between (co-)Heyting algebras and
so-called Esakia spaces with p-morphisms, from which we will pick up
Fact~\ref{fa:duality} below. We first need a notation. Given an
element $x$ in a poset $I$ and a subset $X$ of $I$ let
\begin{align*}
  x^\downarrow = \big\{y\in I\tq y\leq x\big\} & &  X\downarrow = \bigcup_{x\in I} x^\downarrow.
\end{align*}
The dual notation $x^\uparrow$ and $X\uparrow$ is defined accordingly. The family
$\cD^\downarrow(I)$ of decreasing subsets of $I$ (that is the sets $X\subseteq I$ such that
$X\downarrow=X$) are the closed sets of a topology on $I$, hence a co-Heyting
algebra with respect to the following operations.
\begin{align*}
  X \join Y =X\cup Y & & X\meet Y = X\cap Y & & X-Y = (X\setminus Y)\downarrow
\end{align*}
The $\join$\--irreducible elements of $\cD^\downarrow(I)$ are precisely the sets
$x^\downarrow$ for $x\in I$.

\begin{fact}\label{fa:duality}
  Let $L$ be a finite co-Heyting algebra and $\cI$ an ordered set.
  Assume that there is a surjective increasing map $\pi:\cI\to\cI(L)$ 
  such that $\pi(x^\uparrow)\subseteq \pi(x)\uparrow$ for every $x\in\cI(L)$.
  Then there exists an $\ltc$\--embedding $\varphi$ of $L$ into
  $\cD^\downarrow(\cI)$ such that\footnote{Note that the composition $\pi\circ\varphi$ is
  not defined. In this proposition $\varphi(a)$ is a decreasing subset of
  $\cI$ and $\pi(\varphi(a))=\{\pi(\xi)\tq\xi\in\varphi(a)\}$.} $\pi(\varphi(a))=a^\downarrow\cap\cI(L)$ for
  every $a \in L$. 
\end{fact}

\subsection{(Sub)scaled lattices.}
Recall that $\lsc=\llat\cup\{-,\pc^i\}_{i\in\NN}=\ltc\cup\{\pc^i\}_{i\in\NN}$ where
$\{\pc^i\}_{i\in\NN}$ is a family of new unary function symbols. With the
examples of the introduction in mind, we define the {\df sc\--dimension}
of a non-zero element $a$ of an $\lsc$\--structure $L$ as 
\[
\scdim a = 
\min\bigl\{k\in\NN\tq a=\jjoin_{0\leq i\leq k} \pc^i(a)\bigr\}.
\]
Of course this is defined only if $\scdim a=\jjoin_{0\leq i\leq k} \pc^i(a)$, for
some $k$. If it is not defined we let $\scdim a=+\infty$, and by convention 
$\scdim\ZERO=-\infty$. The {\df sc\--dimension of $L$}, denoted
$\scdim(L)$, is the sc\--dimension of $\UN_L$. In general the
dimension of an element in a co-Heyting algebra is not preserved by
$\ltc$\--embeddings. On the contrary the sc\--dimension of an
element is obviously preserved by $\lsc$\--embeddings, and this
is the ``raison d'\^etre'' of this structure.

A {\df $d$\--subscaled lattice} is an $\lsc$\--structure
whose $\ltc$\--reduct is a co-Heyting algebra and which satisfies the
following list of axioms:
\newcommand{\ssref}[1]{$\rm SS_{\ref{#1}}$}
\newcommand{\ssun}{$\rm SS_1^d$}
\newcommand{\ssdeux}{$\rm SS_2^d$}
\begin{list}{$\rm\bf SS_{\theenumi}$:}{\usecounter{enumi}}
  \refstepcounter{enumi}
  \refstepcounter{enumi}
  \item[$\rm\bf SS_1^d$:]
    $\displaystyle \jjoin_{0\leq i\leq d} \pc^i(a)=a$\quad
    and\quad $\forall i>d$, \ $C^i(a)=\ZERO$.
  \item[$\rm\bf SS_2^d$:]
%     \label{SS:Ck(Ci(a))}%
    $\forall I\subseteq\{0,\dots,d\},\ \forall k$:
    \[
    \pc^k\Bigl(\jjoin_{i\in I} \pc^i(a)\Bigr)=
    \left\{
    \begin{array}{cl}
    \ZERO  & \hbox{if } k\nin I\\
    \pc^k(a) & \hbox{if } k\in I\\
    \end{array}
    \right.
    \]
  \item\label{SS:Ck(a-union-b)}%
    $\forall k\geq \max(\scdim(a),\scdim(b))$,\quad
    $\pc^k(a\join b)=\pc^k(a)\join \pc^k(b)$
  \item\label{SS:Ci-inter-Cj}%
    $\forall i\neq j$,\quad 
    $\scdim \left(\pc^i(a) \meet \pc^j(a) \right) < \min(i,j)$
  \item\label{SS:Ck(a)-moins-b=Ck(a)}%
    $\forall k\geq \scdim(b)$,\quad $\pc^k(a)- b = \pc^k(a)- \pc^k(b)$.
  \item\label{SS:ll-scdim}%
    If $b\ll a$ then $\scdim b<\scdim a$.
\end{list}
\newcounter{sauvenumi}
\setcounter{sauvenumi}{\theenumi}
It is a {\df $d$\--scaled lattice} if it satisfies in addition the
following property: 
\newcommand{\scref}[1]{$\rm SC_{\ref{#1}}$}
\begin{list}{$\rm\bf SC_{\theenumi}$:}{\usecounter{enumi}}
  \setcounter{enumi}{-1}
  \item
    \label{SC:scdim-dim}
    $\scdim a=\dim a$
\end{list}

All the geometric $\lsc$\--structures in $\sdef(\cK,d)$ or
$\szar(K,d)$ (defined after Example~\ref{ex:Ldef}) are $d$\--scaled
lattices (see Proposition~\ref{pr:geom-scaled}). However
\scref{SC:scdim-dim} does not follow from the other axioms as the
following example shows. 

\begin{example}\label{ex:Lnoeth}
Let $L$ be an arbitrary noetherian lattice, and $D\colon\cI(L)\to\{0,\dots,d\}$ be a
strictly increasing map. For every $a,b\in L$, if $\cC(a)$
denotes the set of all $\join$\--irreducible components of $a$, let
\begin{align*}
  a-b &= \jjoin\{c\in\cC(a)\tq c\not\leq b\},  \\
  (\forall k)\qquad \pc_D^k(a) &= \jjoin\{c\in\cC(a)\tq D(c)=k\}. 
\end{align*}
This is a typical example of a $d$\--subscaled lattice
in which the sc\--dimension does not coincide with the dimension,
except of course if $D(a)=\dim_L a$ for every $a\in\cI(L)$. Conversely,
every noetherian (in particular every finite) $d$\--subscaled lattice is of
that kind.
\end{example}

We call {\df (sub)scaled lattices} the $\lsc$\--structures whose
$\lsc$\--reduct is a $d$\--(sub)scaled for some $d\in\NN$. Of course this
is not an elementary class. On the contrary, for any fixed $d\in\NN$,
\ssun{} to \ssref{SS:ll-scdim} are expressible by a
universal formula and \scref{SC:scdim-dim} by a first order formula
in $\lsc$, hence $d$\--scaled (resp. $d$\--subscaled) lattices form
elementary class. As the
terminology suggests, we will see that $d$\--subscaled lattices are
precisely the $\lsc$\--substructures of $d$\--scaled lattices. 

\begin{remark}
  \ssun{} to \ssref{SS:Ck(a)-moins-b=Ck(a)} are actually expressible
  by equations in $\lsc$, hence define a
  variety\footnote{Is this the
    variety generated by $d$\--scaled lattices? This question might be
    of importance for further developments in non-classical
    logics.}
  (in the sense of universal algebra). This is clear for \ssun{} and
  \ssdeux{}. The other ones can then be written as follows. 
  \begin{description}
    \item[\ssref{SS:Ck(a-union-b)}]
      $(\forall k\geq l),\quad \pc^k(\jjoin_{i\leq l}\pc^l(a)\join\pc^l(b))
      = \pc^k(\jjoin_{i\leq l}\pc^l(a))\join \pc^k(\jjoin_{i\leq l}\pc^l(b))$
    \item[\ssref{SS:Ci-inter-Cj}] 
      $\forall i>j,\quad \pc^i(a) \meet \pc^j(a) 
      = \jjoin_{k<j}\pc^k(\pc^i(a) \meet \pc^j(a))$
    \item[\ssref{SS:Ck(a)-moins-b=Ck(a)}] 
      $\forall k\geq0,\quad \pc^k(a) - \jjoin_{l\leq k}\pc^l(b) 
      = \pc^k(a) - \pc^k(b)$
  \end{description}
\end{remark}

By analogy with our guiding geometric examples, we say that an element
$a$ in a $d$\--subscaled lattice is {\df $k$\--sc\--pure} if
\[
  \forall b\in L\ (a-b\neq\ZERO\donc \scdim(a-b )=k). 
\]
We will see that $a$ is $k$\--sc\--pure if and only
if $a=\pc^k(a)$ (this is \ssref{SS:Ck(a)-k-pure} in
Section~\ref{se:embedd-subsc-latt}). Then either $a=\ZERO$ or $\scdim
a=k$. In the latter case we say that $a$ {\df has pure sc\--dimension
$k$}. For any $a$, the element $\pc^k(a)$ is called the {\df
$k$\--sc\--pure component} of $a$, or simply its {\df $k$\--pure
component} if $L$ is a scaled lattice. By construction these notions
coincide with their geometric counterparts in $\sdef(\cK,d)$ and
$\szar(K,d)$. 
\\

The following notation will be convenient in induction
arguments. If $\cL$ is any of our languages $\llat$, $\ltc$
or $\lsc$ we let $\cL^*=\cL\setminus\{\UN\}$. Given an $\cL$\--structure $L$
whose reduct to $\llat$ is a lattice, for any $a\in L$ we let 
\[
  L(a)=\{b\in L\tq b\leq a\}. 
\]
$L(a)$ is a typical example of $\cL^*$\--substructure of $L$.

\section{Basic properties and embeddings}
\label{se:embedd-subsc-latt}

The next properties follow easily from the axioms of $d$\--subscaled
lattices.

\begin{list}{$\mathbf{SS_{\theenumi}}$:}{\usecounter{enumi}}
  \setcounter{enumi}{\thesauvenumi}
  \item
    \label{SS:scdim-max-k}
    $\scdim a=\max\{k\tq \pc^k(a)\neq\ZERO\}$.\\
    In particular $\forall k,\ \scdim \pc^k(a)=k\Lssi\pc^k(a)\neq\ZERO$.
  \item
    \label{SS:Ck-a-moins-b}
    $\forall k\geq \scdim(a)$, $\scdim(b\meet a) <k \Ldonc\pc^k(a)- b = \pc^k(a)$.
  \item
    \label{SS:scdim-union}%
    $\scdim a\join b =\max(\scdim a, \scdim b)$.\\
    In particular $b\leq a\donc\scdim b\leq \scdim a$.
  \item
    \label{SS:main-composante}%
    $\forall k\geq \scdim(a)$, $\pc^k(a)$ is the largest $k$\--sc\--pure
    element smaller than $a$. 
  \item
    \label{SS:scdim-inf-dim}
    $\dim a\leq \scdim a$.
  \item
    \label{SS:a-moins-Cd(a)}
    $\displaystyle
    \forall I\subseteq\{0,\dots,d\},\quad
    a-\jjoin_{i\in I}\pc^i(a)=\jjoin_{i\nin I}\pc^i(a)
    $.\\
    In particular $\displaystyle\scdim\big(a-\jjoin_{i\geq k}\pc^i(a)\big)<k$. 
  \item
    \label{SS:Ck(a)-k-pure}
    $\displaystyle
    \forall k,\quad
    \pc^k(a)=a \Lssi
    \forall b\ \big(a-b\neq\ZERO\donc\scdim a-b = k\bigr)
    $.\\
    That is $a$ is $k$\--sc\--pure if and only if $a=\pc^k(a)$. \\
    In particular if $a$ is $\join$\--irreducible then $a$ is sc-pure by
    \ssun{}.
\end{list}

\begin{proof} (Sketch)
\ssref{SS:scdim-max-k} follows from 
\ssun{} and \ssdeux{};
\ssref{SS:Ck-a-moins-b} from \ssref{SS:Ck(a)-moins-b=Ck(a)}
and \ssref{SS:scdim-max-k};
\ssref{SS:scdim-union} from 
\ssdeux{}, \ssref{SS:Ck(a-union-b)} 
and \ssref{SS:scdim-max-k}; 
\ssref{SS:main-composante} from
\ssdeux{} and \ssref{SS:Ck(a-union-b)};
\ssref{SS:scdim-inf-dim} from \ssref{SS:ll-scdim} by
Fact~\ref{fa:ldim-et-TCdim}.
Only the two last properties require a little effort. 
\smallskip

{\bf \ssref{SS:a-moins-Cd(a)}:} 
For every $l\in I$, $\pc^l(a)\leq \jjoin_{i\in I}\pc^i(a)$
hence $ \pc^l(a)-\jjoin_{i\in I}\pc^i(a)=\ZERO$. 
On the other hand for every $l\nin I$ and 
every $i\in I$, $\pc^l(a)-\pc^i(a)=\pc^l(a)$ 
by \ssref{SS:Ci-inter-Cj} and \ssref{SS:Ck(a)-moins-b=Ck(a)}. 
So $\pc^l(a)-\jjoin_{i\in I}\pc^i(a)=\pc^l(a)$ 
by \tcref{TC:x-(y-union-z)}. Finally by 
\ssun{} and \tcref{TC:(x-union-y)-z},
\[
a-\jjoin_{i\in I}\pc^i(a)
=\jjoin_{l\leq d}\Bigl(\pc^l(a)-\jjoin_{i\in I}\pc^i(a)\Bigr)
=\jjoin_{l\nin I}\pc^l(a).
\]

{\bf \ssref{SS:Ck(a)-k-pure}:}
Assume that $a=\pc^k(a)$ and $a-b\neq\ZERO$ for some $b$. Then $\scdim a
=k$ so $\scdim(a-b)\leq k$ and $\scdim(a\meet b)\leq k$ by
\ssref{SS:scdim-union}. Since $a=(a-b)\join(a\meet b)$ it follows by
\ssref{SS:Ck(a-union-b)} that
\[
  \pc^k(a)=\pc^k\big((a-b)\join(a\meet b)\big)=\pc^k(a-b)\join\pc^k(a\meet b).
\]
By assumption $\pc^k(a)=a$, and
$\pc^k(a\meet b)\leq a\meet b$ by \ssun{}. 
So $a\leq \pc^k(a-b)\join (a\meet b)\leq \pc^k(a-b)\join b$ which implies that
$a-b\leq \pc^k(a-b)$. In particular $\pc^k(a-b)\neq\ZERO$. Since
$\scdim(a-b)\leq k$ it follows that $\scdim(a-b)=k$ by 
\ssref{SS:scdim-max-k}.

Conversely assume that $a\neq\pc^k(a)$ (hence $a\neq\ZERO$). For
$b=\pc^k(a)$ we then have $a-b\neq\ZERO$ on one hand and $\scdim(a-b)\neq k$
by \ssref{SS:scdim-max-k} on the other hand, because
$\pc^k(a-b)=\ZERO$ by \ssref{SS:a-moins-Cd(a)} and \ssdeux{}.
\end{proof}

\begin{proposition}\label{pr:scaled-structure-definissable}
The $\lsc$\--structure of a $d$\--scaled lattice $L$ is 
uniformly definable in the $\llat$\--structure of $L$. 
In particular it is uniquely determined by this 
$\llat$\--structure. 
\end{proposition}

\begin{proof}
Clearly the $\ltc$\--structure is an extension by definition of the
lattice structure of $L$. For every positive integer $k$ the class of
$k$-pure elements is uniformly definable, using the definability of
$\ll$ and Fact~\ref{fa:ldim-et-TCdim}. Then  so is the function
$\pc^k$ for every $k$, by decreasing induction on $k$. Indeed by
\ssref{SS:main-composante} and \ssref{SS:a-moins-Cd(a)}, $\pc^k(a)$ is
the largest $k$\--pure element $c$ such that $c\leq a-\jjoin_{i>k}\pc^i(a)$.
\end{proof}

We need a reasonably easy criterion for an $\llat$\--embedding of
subscaled lattices to be an $\lsc$\--embedding. In the special case of
a noetherian\relax
\footnote{Although we won't use it, let us mention that in the general
  case of an $\llat$\--embedding $\varphi:L\to L'$ between arbitrary subscaled
  lattices, one may easily derive from
  Proposition~\ref{pr:CNS-de-lsc-plongement}, by means of the Local
  Finiteness Theorem~\ref{th:TCS-treillis-type-fini} and the
  model-theoretic compactness theorem, that $\varphi$ is an
  $\lsc$\--embedding if and only if it preserves sc\--dimension and
  sc\--purity, that is for every $a\in L$ and every $k\in\NN$, $\pc^k(a)=a
  \Rightarrow \pc^k(\varphi(a))=\varphi(a)$.
}
embedded lattice, it is given by
Proposition~\ref{pr:CNS-de-lsc-plongement} below, whose proof will
use the following characterisation of sc\--pure components.

\begin{proposition}\label{pr:comp-pures}
  Let $L$ be a subscaled lattice and $a,a_0,\dots,a_d\in L$ be such that
  $a=\jjoin_{i\leq d}a_i$, each $a_i$ is $i$\--sc\--pure and $\scdim(a_i\meet
  a_j)<\min(i,j)$ for every
  $i\neq j$. Then $\pc^i(a)=a_i$ for every $i$.
\end{proposition}

\begin{proof}
  Note first that $\pc^k(a_k)=a_k$ for every $k\leq d$ by
  \ssref{SS:Ck(a)-k-pure}. Hence for every $k\leq i\leq d$ we have by
  \ssref{SS:Ck(a-union-b)} and \ssdeux{}
  \begin{equation}
    \pc^i\Big(\jjoin_{k\leq i}a_k\Big)=\jjoin_{k\leq i}\pc^i(a_k)=\jjoin_{k\leq
    i}\pc^i\big(\pc^k(a_k)\big)=\pc^i(a_i)=a_i.
    \label{eq:pure1}
  \end{equation}
  In particular $\pc^d(a)=a_d$. Now assume that for some $i<d$ we have
  proved that $\pc^j(a)=a_j$ for $i<j\leq d$. By
  \ssref{SS:a-moins-Cd(a)} and \ssdeux{} we then have
  \begin{equation}
    \pc^i\big(a-\jjoin_{j>i}a_j\big)
    =\pc^i\big(a-\jjoin_{j>i}\pc^j(a)\big)
    =\pc^i\big(\jjoin_{j\leq i}\pc^j(a)\big)
    =\pc^i(a).
    \label{eq:pure2}
  \end{equation}
  On the other hand $a-\jjoin_{j>i}a_j=\jjoin_{k\leq d}(a_k-\jjoin_{j>i}a_j)$ by
  \tcref{TC:(x-union-y)-z}. For 
  $k>i$ obviously $a_k-\jjoin_{j>i}a_j=\ZERO$. For $k<i$, $a_k=\pc^k(a)$ and
  $a_j=\pc^j(a)$ imply that $\scdim(a_k\meet
  a_j)<k$ for $j>i$ by \ssref{SS:Ci-inter-Cj}. Hence $a_k-a_j=a_k$ by
  \ssref{SS:Ck-a-moins-b} and finally $a_k-\jjoin_{j>i}a_j=a_k$ by
  \tcref{TC:x-(y-union-z)}, so 
  \begin{equation}
    a-\jjoin_{j>i}a_j=\jjoin_{k\leq i}a_k.
    \label{eq:pure3}
  \end{equation}
  By (\ref{eq:pure1}), (\ref{eq:pure2}), (\ref{eq:pure3})
  we conclude that $\pc^i(a)=a_i$. The result follows for
  every $i$ by decreasing induction.
\end{proof}

\begin{proposition}%
  \label{pr:CNS-de-lsc-plongement}%
  Let $L_0$ be a noetherian subscaled lattice, $L$ a subscaled
  lattice, and $\varphi:L_0\to L$ an $\llat$\--embedding such that for every
  $a\in\cI(L_0)$, $\varphi(a)$ is sc\--pure and has the same sc\--dimension as
  $a$. Then $\varphi$ is an $\lsc$\--embedding.
\end{proposition}

\begin{remark}
Clearly the same statement remains true with $\llat$ and $\lsc$ 
replaced respectively by $\llat^*$ and $\lsc^*$. 
We will freely use these variants.
\end{remark}

\begin{proof}
We have that $L_0$ and $L$ are $d$\--subscaled lattices for some $d\in\NN$.
Given $a\in L_0$ and $k$ a non-negative integer, 
we first check that $\varphi(\pc^k(a))$ is $k$\--sc\--pure. Note that every
$\join$\--irreducible component $c$ of $\pc^k(a)$ in $L_0$ has sc\--pure
dimension $k$. Indeed $\pc^k(a)$ is $k$\--sc\--pure by
\ssref{SS:Ck(a)-k-pure}, and $c=\pc^k(a)-b\neq\ZERO$ where $b$ is the
join of all the other $\join$\--irreducible components of $\pc^k(a)$,
hence $\scdim(c)=\scdim(\pc^k(a)-b)=k$. Moreover $c$ is sc\--pure
because it is $\join$\--irreducible, hence $c$ is $k$\--pure. By our
assumption on $\varphi$ it follows that $\varphi(c)$ is $k$\--sc\--pure. Every
finite union of $k$\--sc\--pure elements being $k$\--sc\--pure by
\ssref{SS:Ck(a-union-b)}, it follows that
\begin{equation}
  \varphi\big(\pc^k(a)\big)\mbox{ is $k$-sc-pure.} \label{eq:phi-k-pure}
\end{equation}

Now for every $l\neq k$ we have $\scdim(\pc^k(a)\meet\pc^l(a))<\min(k,l)$ 
by \ssref{SS:Ci-inter-Cj}. It follows that each 
$\join$\--irreducible component $c$ of $\pc^k(a)\meet\pc^l(a)$ 
has sc\--dimension strictly less than $\min(k,l)$, hence so does 
$\varphi(c)$ by assumption. By \ssref{SS:scdim-union} we conclude that 
\begin{equation}
   \scdim\big(\pc^k(a)\meet\pc^l(a)\big)<\min(k,l)\quad(\forall l\neq k).
   \label{eq:dim-de-phi-ak-inter-al}
\end{equation}

We have that $\varphi(a)=\jjoin_{k\leq d}\varphi(\pc^k(a))$ by \ssun{} and
because $\varphi$ is an $\llat$\--embedding. By (\ref{eq:phi-k-pure}),
(\ref{eq:dim-de-phi-ak-inter-al}) and Proposition~\ref{pr:comp-pures}
it follows that $\pc^k(\varphi(a))=\varphi(\pc^k(a))$ for every $k\leq d$. 
Since $\varphi$ is injective, this implies 
by \ssref{SS:scdim-max-k} that for every $a\in L_0$ 
\begin{equation}\label{eq:phi-preserve-scdim}
  \scdim a=\scdim\varphi(a).
\end{equation}

It only remains to check that $\varphi(a-b)=\varphi(a)-\varphi(b)$ for every $a,b\in L_0$. 
By \tcref{TC:(x-union-y)-z}, replacing if necessary $a$ 
by its $\join$\--irreducible components, we may assume w.l.o.g. 
that $a$ itself is $\join$\--irreducible in $L_0$. This implies 
that $a=\pc^k(a)$ for some $k$, hence $\varphi(a)$ is $k$\--sc\--pure by
assumption on $\varphi$. 
It then remains two possibilities for $a-b$:
\begin{itemize}
  \item 
    If $b\geq a$ then $\varphi(b)\geq\varphi(a)$, hence $a-b=\ZERO$ and
    $\varphi(a)-\varphi(b)=\ZERO$, so $\varphi(a-b)=\varphi(\ZERO)=\ZERO=\varphi(a)-\varphi(b)$.
  \item
    Otherwise $b\meet a<a$ hence $a-b=a$ by
    \tcref{TC:a=(a-inter-b)-union-(a-b)}. So we have
    to prove that $\varphi(a)-\varphi(b)=\varphi(a)$. By \ssref{SS:ll-scdim} 
    $\scdim b\meet a<\scdim a$, hence $\scdim(\varphi(b)\meet\varphi(a))<\scdim(\varphi(a))=k$ by
    (\ref{eq:phi-preserve-scdim}). Since $\varphi(a)$ is $k$\--sc\--pure it
    follows that $\varphi(a)-\varphi(b)=\varphi(a)$ by \ssref{SS:Ck-a-moins-b}. 
\end{itemize}
\end{proof}

\begin{corollary}\label{co:ss-lsc-structure}
  Let $L_0$ be a noetherian lattice embedded in a subscaled lattice
  $L$. Assume that every $b'<b\in\cI(L_0)$ are sc-pure in $L$ and
  $\scdim b'<\scdim b$ in $L$. Then the restrictions to $L_0$ of the
  $\lsc$\--operations ``$-$'' and ``$C_i$'' of $L$ turn $L_0$ into an
  $\lsc$\--substructure which is a subscaled lattice.
\end{corollary}

\begin{proof}
The assumptions imply that the map $D:a\mapsto\scdim a$ 
is a strictly increasing map from $\cI(L_0)$ to $\NN$. Endow $L_0$ with
the structure of subscaled lattice determined by $D$ as in
Example~\ref{ex:Lnoeth}. By construction the inclusion map from $L_0$
to $L$ is an $\llat$\--embedding which preserves the sc-purity and
sc-dimension of every $b\in\cI(L_0)$, hence is an $\lsc$\--embedding by
Proposition~\ref{pr:CNS-de-lsc-plongement}. 
\end{proof}

\section{Local finiteness}
\label{se:local-finiteness}

We prove in this section that every finitely generated subscaled 
lattice is finite. This result is far from obvious, due to the lack 
of any known normal form for terms in $\lsc$. It contrasts with 
the general situation in co-Heyting algebras, which can be both infinite and 
generated by a single element. Our main ingredient, which explains 
this difference, is the uniform bound given {\it a priori} for the 
sc\--dimension of any element in a given $d$\--subscaled lattice.

\begin{theorem}\label{th:TCS-treillis-type-fini}
Any $d$\--subscaled lattice $L$ generated by $n$ 
elements is finite. More precisely, the cardinality of $\cI(L)$ 
is bounded by the function $\mu(n,d)$ defined by
\[
  \mu(n,d) = 2^n+ \mu(2^{n+1},d-1).
\]
for $d\geq0$, and $\mu(n,d)=0$ for $d<0$.
\end{theorem}

\begin{proof}
The only subscaled lattice of sc-dimension $d<0$ is 
the one-element lattice $\{\ZERO\}$, so the result is trivial in this
case. Assume that $d\geq 0$ and that the result is proved for every 
$d'<d$ and every non-negative integer $n$. 

Let $L$ be a subscaled lattice of sc\--dimension $d$ generated by
elements $x_1,\dots,x_n$. Let $\Omega_n$ be the family of all subsets of
$\{1,\dots,n\}$ (so $\Omega_0=\{\emptyset\}$). For every $I\in \Omega_n$ let $I^c=\Omega_n\setminus I$ and
\[
  y_I = \Bigl(\mmeet_{i\in I} x_i\Bigr)
        - \Bigl(\jjoin_{i\in I^c} x_i\Bigr),
  \qquad
  z_I =\pc^d(y_I).
\]
The family of all $\cY_I = \bigcap_{i\in I} P(x_i) \cap \bigcap_{i\in I^c} P(x_i)^c$
is a partition of $\Sp(L)$. Indeed the $\cY_i$'s are the atoms 
of the boolean algebra generated in the power set $\cP(\Sp(L))$ 
by the $P(x_i)$'s. Moreover each $P(y_I)$ is the 
topological closure $\overline{\cY}_I$ of $\cY_I$ in $\Sp(L)$ 
hence for every $x\in L$
\[
  P(x)=\bigcup_{I\in\Omega_n}P(x)\cap \cY_I
  \subseteq\bigcup_{I\in\Omega_n}P(x)\cap \overline{\cY}_I
  =P\Bigl(\jjoin_{I\in\Omega_n}x\meet y_I\Bigr).
\]
So $x\leq \jjoin\limits_{I\in\Omega_n}(x\meet y_I)$ 
by Stone's duality. The reverse inequality being 
obvious we have proved that
\begin{equation}
  \label{eq:x-et-les-xi}
  \forall x\in L,\quad x = \jjoin\limits_{I\in\Omega_n}(x\meet y_I).
\end{equation}
In particular \ssref{SS:Ck(a-union-b)} also gives
\begin{equation}
  \label{eq:Cd-1-egale-union-des-Cd-zI}
  \pc^d(\UN)
  =\pc^d\Bigl(\jjoin\limits_{I\in\Omega_n}y_I\Bigr)
  =\jjoin\limits_{I\in\Omega_n}z_I.
\end{equation} 
For every $I\neq J\in\Omega_n$, if for example $I\not\subseteq J$ 
choose any $i\in I\setminus J$ and observe that 
$y_I\leq x_i$ and $y_J\leq \UN-x_i$ so $y_I\meet y_J\ll\UN-x_i$
by \tcref{TC:(x-y)-y}. By \ssref{SS:ll-scdim} and the 
$d$\--sc\--purity of the $z_I$'s it follows that
\begin{equation}
  \label{eq:zI-meet-zJ}
  \scdim z_I\meet z_J<d\hbox{\quad hence\quad}z_I-z_J=z_I.
\end{equation}
It follows from \ssref{SS:scdim-union}, \ssref{SS:a-moins-Cd(a)} 
and (\ref{eq:zI-meet-zJ}) above, that the element
\[
  a=\big(\UN-\pc^d(\UN)\big)\join
  \Big(\jjoin\limits_{I\in\Omega_n}(y_I-z_I)\Big)\join
  \Big(\jjoin\limits_{I\neq J\in\Omega_n}(z_I\meet z_J)\Big) 
\]
has sc\--dimension strictly smaller than $d$. So the induction
hypothesis applies to the $\lsc$\--substructure $L_0^-$ of $L(a)$
generated by the ${(y_I-z_I)}$'s and the ${(z_I\meet a)}$'s: $L_0^-$ is
finite, with at most $\mu(2|\Omega_n|,d-1)$ $\join$\--irreducible
elements. Note that $L_0^-$ is an $\lsc^*$\--substructure of $L$
(recall that $\lsc^*=\lsc\setminus\{\UN\}$). Finally let $L_1$ be the upper
semi-lattice generated in $L$ by $L_0^-\cup\{z_I\}_{I\in\Omega_n}$. By
construction $L_1$ is finite and $\cI(L_1)\subseteq\cI(L_0^-)\cup\{z_I\}_{I\in\Omega_n}$,
so $|\cI(L_1)|\leq 2^n+\mu(2^{n+1},d-1)=\mu(n,d)$. It is then sufficient to
show that $L_1=L$. 

We first prove that $L_1$ is a lattice.
By~(\ref{eq:Cd-1-egale-union-des-Cd-zI}), $\pc^d(\UN)\join a 
=\jjoin_{I\in\Omega_n}z_I\join a\in L_1$ hence $\UN=\pc^d(\UN)\join a\in L_1$. For every
$I\in\Omega_n$ and every $b\in L_0^-$, $z_I\meet b=(z_I\meet a)\meet b\in L_0^-$. For
every $I\neq J\in\Omega_n$, $z_I\meet z_J=(z_I\meet a)\meet(z_J\meet a)\in L_0^-$. So by the
distributivity law, $L_1$ is a sublattice of $L$. 

In order to conclude that $L_1$ is an $\lsc$\--substructure of $L$, by
Corollary~\ref{co:ss-lsc-structure} it only remains to check that for
every $b'<b\in\cI(L_1)$, $b$ is sc-pure in $L$ and $\scdim b'<\scdim b$
in $L$. Since $\cI(L_1)\subseteq\cI(L_0^-)\cup\{z_I\}_{I\in\Omega_n}$ we can distinguish
two cases.

{\it Case 1}: $b\in\cI(L_0^-)$. Then $b$ is sc-pure in $L_0^-$ by
\ssref{SS:Ck(a)-k-pure} hence
also in $L$ since $L_0^-$ is an $\lsc^*$\--substructure of $L$.
Similarly $b'\ll b$ in $L_0^-$ by \tcref{TC:a=(a-inter-b)-union-(a-b)}
that is $b-b'=b$ in $L_0^-$ hence also in $L$. Thus $b'\ll b$ in $L$
which implies that $\scdim b' < \scdim b$ in $L$ by
\ssref{SS:ll-scdim}. 

{\it Case 2}: $b=z_I$ for some $I\in\Omega_n$. Then $b=\pc^d(y_I)$ is sc-pure
in $L$ and $\scdim b=d$. If
$b'=z_J$ for some other $J\in\Omega_n$ then on one hand $\scdim(b')=d$ and on
the other hand $I\neq J$ hence $b'=b'\meet b=z_I\meet z_J$ has sc-dimension $<d$ by
(\ref{eq:zI-meet-zJ}), a contradiction. So necessarily $b'\in\cI(L_0^-)$, in
particular $b'\leq a$ hence $\scdim(b')\leq\scdim(a)<d$

So $L_1$ is indeed an $\lsc$\--substructure of $L$. Finally every
$y_I=(y_I-z_I)\join z_I\in L_1$ and (\ref{eq:x-et-les-xi})
gives, for every $i\leq n$,
\[
x_i=\jjoin_{I\in\Omega_n}x_i\meet y_I
\leq \jjoin_{{I\in\Omega_n}\atop{i\in I}}y_I\leq x_i.
\]
So equality holds, hence each $x_i\in L_1$, which finally proves that 
$L_1=L$. 
\end{proof}

\begin{corollary}
\label{cor:structures-a-n-generateurs}
For every $n,d$ there are finitely many non-isomorphic 
subscaled lattices of sc\--dimension $d$ generated by $n$ elements.
\end{corollary}

\begin{proof}
Any such subscaled lattice $L$ is finite, with 
$|\cI(L)|\leq \mu(n,d)$ by Theorem~\ref{th:TCS-treillis-type-fini}. 
Clearly there are finitely many non-isomorphic lattices such that 
$|\cI(L)|\leq \mu(n,d)$ and each of them admits finitely 
many non-isomorphic $\lsc$\--structures of $d$\--subscaled lattices. 
\end{proof}

\section{Linear representation}
\label{se:lin-rep}

In this section we prove that the theory of $d$\--subscaled lattices
is the universal theory of various natural classes of geometric
$d$\--scaled lattices, including $\sdef(\cK,d)$ in
Example~\ref{ex:Ldef} as well as $\szar(K,d)$. The argument is based
on an elementary representation theorem for $d$\--subscaled lattices,
combined with the local finiteness result of
Section~\ref{se:local-finiteness}.

Given an arbitrary field $K$, a non-empty linear variety $X\subseteq K^m$ is
determined by the data of an arbitrary point $P\in X$ and the vector
subspace $\vect X$ of $K^m$, {\it via} the relation $X=P+\vect X$ (the
orbit of $P$ under the action of $\vect X$ by translation). We call
$X$ a {\df special linear variety} (resp. a {\df special linear set})
if $X$ is a linear variety such that $\vect X$ is generated by a
subset of the canonical basis of $K^m$ (resp. if $X$ is a finite union
of special linear varieties). The family $\llin(X)$ of all special
linear subsets of $X$ is the family of closed sets of a noetherian
topology on $X$, hence a noetherian lattice. For every $A\in\llin(X)$ we
let $D(A)$ be the dimension of $A$ in the sense of linear algebra.
This endows $\llin(X)$ with a natural
structure of scaled lattice as in Example~\ref{ex:Lnoeth}.

\begin{remark}\label{re:sclin-inclus-dans-scdef-et-sczar}
  For every $A\in\llin(X)$, if $K$ is infinite then $\scdim
  A=\dim_{\llin(X)} A=$ the dimension of $A$ as defined in linear
  algebra. It coincides with the Krull dimension as well. If moreover
  $A$ is $\join$\--irreducible in $\llin(X)$ then it is pure dimensional,
  hence it is sc\--pure both in $\llin(X)$ and $\lzar(X)$. By
  Proposition~\ref{pr:CNS-de-lsc-plongement} it follows that
  $\llin(X)$ is an $\lsc$\--substructure of $\lzar(X)$. Similarly if
  $\cK$ is a tame expansion of a topological field 
  then $\llin(X)$ is an $\lsc$\--substructure of $\ldef(X)$. 
\end{remark}

In what follows $K^m$ is identified with $K^m\times\{0\}^{r}\subseteq K^{m+r}$. The
very easy result below prepares the proof of
Proposition~\ref{pr:linear-repres}.

\begin{proposition}%
\label{pr:treillis-des-lineaires-speciaux}%
For every two special linear sets $C\subseteq B\subseteq K^m$ and every non-negative integer
$n\geq\dim C$ there exists a special linear set $A\subseteq K^{m+n}$ of pure
dimension $n$ such that $A\cap B=C$. 
\end{proposition}

\begin{proof}
The result being rather trivial if $C$ is empty, we can assume
w.l.o.g. that $C\neq\emptyset$. Let $(e_1,\dots,e_{m+n})$ be the
canonical basis of $K^{m+n}$. If $I$ is a subset of $\{1,\dots,m+n\}$ we let
$\vect E(I)$ denote the vector space generated in $K^n$ by $(e_i)_{i\in
I}$. Decompose $C$ as a union of special linear varieties $C_1,\dots,C_p$,
and write each $C_i=P_i+\vect E(J_i)$ with $|J_i|=\dim C_i\leq n$. Let
$I_i=J_i\cup\{m+1,\dots,m+n-|J_i|\}$ and $A_i=P_i+\vect E(I_i)$ for every $i\leq
p$. Finally let $A=A_1\cup\dots\cup A_p$. By construction each $A_i$ has pure
dimension $|I_i|=n$, hence $A$ has pure dimension $n$. Clearly each
$A_i\cap K^m=C_i$, hence $A\cap K^m=C$ and {\it a fortiori} $A\cap B=C$.
\end{proof}

\begin{proposition}[Linear representation]\label{pr:linear-repres}
Let $K$ be an infinite field, $d\geq0$ an integer and $L$ a finite
$d$\--subscaled lattice. Then there exists a special linear set $X$
over $K$ of dimension $\leq d$ and an $\lsc$\--embedding $\varphi:L\to\llin(X)$. 
\end{proposition}

\begin{proof}
By induction on the number $r$ of $\join$\--irreducible elements
of an arbitrary $d$\--subscaled lattice $L$, we prove that there
exists an $\lsc^*$\--embedding $\varphi$ of $L$ into $\llin(K^m)$ for some
$m$ depending on $L$. Taking $X=\varphi(\UN_L)$ then gives the conclusion.
Indeed $X$ is a special linear set over $K$, $\dim X=\scdim(\UN_L)\leq d$
because $\varphi$ preserves the sc-dimension, and $\varphi$ is obviously an
$\lsc$\--embedding of $L$ into $\llin(X)$.

If $r=0$ then $L$ is the one-element lattice $\{\ZERO\}$, hence an
$\lsc^*$\--substructure of $\llin(K)$. So, given a fixed $r\geq 1$, we
can assume by induction that the result is proved for $r-1$. Let $L$ be a
$d$\--subscaled lattice with $\join$\--irreducible elements $a_1,\dots,a_r$.
Let $a=a_r$ and $b=\jjoin_{1\leq i<r}a_i$.

Renumbering if necessary we may assume that $a_r$ is maximal among the
$a_i$'s. By maximality, the
$\join$\--irreducible elements of $L(b)$ are $a_1,\dots,a_{r-1}$. Let $c=a\meet b$
and $\varphi$ an $\lsc^*$\--embedding of $L(b)$ into some $\llin(K^m)$ given
by the induction hypothesis. Since $a$ is $\join$-irreducible in $L$ it is
sc\--pure. Moreover $c\leq a$ by \tcref{TC:a=(a-inter-b)-union-(a-b)},
hence $a$ has pure sc\--dimension $n$ for some $n\geq\scdim(c)$ by
\ssref{SS:ll-scdim}. Let $B,C$ be the respective images of $b,c$ by
$\varphi$. Proposition~\ref{pr:treillis-des-lineaires-speciaux} gives a
special linear set $A\subseteq K^{m+n}$ of pure dimension $n$ such that $A\cap
B=C$. Identifying $K^m$ with $K^m\times\{0\}^n\subseteq K^{m+n}$ turns $\varphi$ into an
$\lsc^*$\--embedding of $L(b)$ into $\llin(K^{m+n})$.

Every  $x\in L$ can be written uniquely as $x_a\join x_b$ with $x_a\in\{\ZERO,a\}$ and
$x_b\in L(b)$ by grouping appropriately the $\join$\--irreducible components
of $x$. So we can let
\[
  \bar\varphi(x)=
  \left\{
  \begin{array}{cl}
    \varphi(x_b)      & \hbox{if $x_a=\ZERO$,}\\
    A\cup\varphi(x_b) & \hbox{if $x_a=a$.}
  \end{array}
  \right.
\]
This is a well-defined $\llat^*$\--embedding of $L$ into
$\llin(K^{m+n})$. Moreover $\bar\varphi$  is an $\lsc^*$\--embedding by
Proposition~\ref{pr:CNS-de-lsc-plongement}. This finishes the
induction. 
\end{proof}

Given an infinite field $K$ and positive integer $d$ 
let $\slin(K,d)$ be the class of $d$-scaled lattices $\llin(X)$ 
with $X$ ranging over the special linear sets over $K$ of dimension 
at most $d$. 

\begin{theorem}
  The universal theories of $\sdef(\cK,d)$ (resp. $\szar(K,d)$,
  $\slin(K,d)$) in the language $\lsc$ are the same for every fixed
  integer $d\geq0$ and every tame expansion $\cK$ of a topological
  field $K$ (resp. for every infinite field $K$). This is the theory
  of $d$\--subscaled lattices.
\end{theorem}

\begin{proof}
  As explained in Section~\ref{se:appendix}, for every such expansion
  $\cK$ of $K$ the good properties of the dimension theory for
  definable sets $X\subseteq K^m$ ensure that $\ldef(X)$ is a $d$\--scaled
  lattice. Obviously the same holds true for $\llin(X)$ and
  $\lzar(X)$. So the universal theory of any of these classes contains
  the theory of $d$\--subscaled lattices. For the converse, thanks to
  Remark~\ref{re:sclin-inclus-dans-scdef-et-sczar} it suffices to
  prove that every $d$\--subscaled lattice $L$ embeds into a model of
  the theory of $\slin(K,d)$. If $L$ is finite this is
  Proposition~\ref{pr:linear-repres}. The general case then follows
  from the model-theoretic compactness theorem, because $L$ is locally
  finite by Theorem~\ref{th:TCS-treillis-type-fini}.
\end{proof}

\section{Minimal extensions}
\label{se:minimal-extensions}

Minimal proper extensions\footnote{When we talk about
  an {\df extension} $L$ of a lattice, a co-Heyting algebra or a subscaled
  lattice $L_0$, it is always understood that $L$ is also a lattice, a
  co-Heyting algebra or a subscaled lattice
  respectively.\label{fn:extension}} of any finite subscaled lattices are
entirely determined by so-called ``SC\--signatures'' (see below).
Since this is a special case of minimal extensions of finite
co-Heyting algebras, we first recall the main results of
\cite{darn-junk-2018} on this subject, and try to reduce to them
as much as possible. 

We need some specific notation and definitions. Given a finite lattice
$L_0$, a $\llat$\--extension $L$, elements $a\in L_0$ and $x\in L$ we
write:
\begin{itemize}
\item
$a^-=\jjoin\{b\in L_0\tq b<a\}$.
\item
$g(x,L_0)=\mmeet\{a\in L_0\tq x\leq a\}$. 
\end{itemize}
Clearly $a\in\cI(L_0)$ if and only if $a^-$ is the unique 
predecessor of $a$ in $L_0$ (otherwise $a^-=a$). 

Assume that $L_0$ and $L$ are co-Heyting algebras (or topologically
complemented lattices, or TC\--lattice for short). A {\df
TC\--signature} in $L_0$ is a triple $(g,H,r)$ where $g\in\cI(L_0)$, $H$
is a set of one or two elements $h_1,h_2\in L_0$ and $r\in\{1,2\}$ are such
that:
\begin{itemize}
  \item
    either $r=1$ and $h_1=h_2<g$;
  \item 
    or $r=2$ and $h_1\join h_2$ is the unique predecessor of $g$.
\end{itemize}
A couple $(x_1,x_2)$ of non-zero  elements
of $L$ is {\df TC\--primitive over} $L_0$ if there is $g\in \cI(L_0)$
such that 
\begin{description}
  \item[P1]
    $g^-\meet x_1$ and $g^-\meet x_2$ belong to $L_0$.
  \item[P2]
    One of the following happens: 
    \begin{enumerate}
        \item
          $x_1=x_2$ and $g^-\meet x_1\ll x_1 \ll g$.
        \item
          $x_1\neq x_2$, $x_1\meet x_2\in L_0$ and $g-x_1=x_2$, $g-x_2=x_1$.
      \end{enumerate}
\end{description}
This implies that each $x_i\nin L_0$, that $g=g(x_1,L_0)=g(x_2,L_0)$
and that the triple $\sigma_\TC(x_1,x_2)=(g,H,r)$ defined as follows is a
SC\--signature in $L_0$, called  the {\df SC\--signature of
$(x_1,x_2)$ in $L_0$}. 
\begin{align*}
  g=g(x_1,L_0) & & H=\{g^-\meet x_1,g^-\meet x_2\} & & r=\card\{x_1,x_2\}
\end{align*}
Finally we say that $L$ is a {\df TC\--primitive extension} of $L_0$
if it is $\ltc$\--generated over $L_0$ by a TC\--primitive couple. For the
convenience of the reader we collect here all the properties of
TC\--signatures and TC\--primitive extensions that we are going
to use. 

We will refer to the $k$\--th item of the next proposition as 
Proposition~\ref{pr:recap-DJ}.$k$.

\begin{proposition}[\cite{darn-junk-2018}]\label{pr:recap-DJ}
  Let $L_0$ be a finite co-Heyting algebra and $L$ an
  $\ltc$\--extension\footnote{See Footnote~\ref{fn:extension}}. 
  \begin{enumerate}
    \item\label{it:recap-Irr}
      (\cite[Theorem~3.3]{darn-junk-2018})
      If $L$ is $\ltc$\--generated over $L_0$ by a TC\--primitive
      tuple $(x_1,x_2)$, then $L$ is exactly the upper semi-lattice
      generated over $L_0$ by $x_1$ and $x_2$. It is a finite
      co-Heyting algebra and one of the following holds: 
      \begin{enumerate}
        \item
          $x_1=x_2$ and $\cI(L)=\cI(L_0)\cup\{x_1\}$.
        \item
          $x_1\neq x_2$ and $\cI(L)=\left(\cI(L_0)\setminus\{g\}\right)\cup\{x_1,x_2\}$.
      \end{enumerate}
    \item\label{it:recap-sgn-ext}
      (\cite[Remark~3.6]{darn-junk-2018})
      The TC\--signatures in $L_0$ and the TC\--primitive extensions of $L_0$
      are in one-to-one correspondence: every TC\--signature in $L_0$
      is the TC\--signature of a TC\--primitive extension, and two
      TC\--primitive extensions of $L_0$ are $\ltc$\--isomorphic over
      $L_0$ if and only if they have the same TC\--signature in $L_0$. 
    \item\label{it:recap-prim-min}
      (\cite[Corollary~3.4]{darn-junk-2018})
      If $L$ is finite, the following are equivalent. 
      \begin{enumerate}
        \item\label{it:recap-min}
          $L$ is a minimal proper extension of $L_0$.
        \item\label{it:recap-prim}
          $L$ is a TC\--primitive extension of $L_0$.
        \item\label{it:recap-card}
          $\card(\cI(L))=\card(\cI(L_0))+1$.
      \end{enumerate}
      As a consequence every finite $\ltc$\--extension $L'$ of $L_0$ is the
      union of a tower of TC\--primitive extensions $L_0\subset L_1\subset\cdots\subset
      L_n=L'$ with $n=\card(\cI(L'))-\card(\cI(L_0))$. 
  \end{enumerate}
\end{proposition}

If $L$ is a TC\--primitive extension of a finite co-Heyting algebra
$L_0$, by Proposition~\ref{pr:recap-DJ}.\ref{it:recap-Irr} it is
$\ltc$\--generated over $L_0$ by a {\em unique} (up to permutation)
TC\--primitive tuple $(x_1,x_2)$. We then call $\sigma_\TC(x_1,x_2)$ the
{\df TC\--signature of $L$ in $L_0$} and denote it $\sigma_\TC(L)$.

Now let $L_0$ be a finite subscaled lattice and $L$ a
$\lsc$\--extension. 
A {\df SC\--signature} in $L_0$ is a triple $\sigma=(g,H,q)$ where
$g\in\cI(L_0)$, $H$ is a set of one or two elements $h_1,h_2\in L_0$ and
$q\in\NN$ are such that:
\begin{itemize}
  \item
    either $\scdim h_1<q<\scdim g$ and $h_1=h_2<g$;
  \item 
    or $q=\scdim g$ and $h_1\join h_2=g^-$.
\end{itemize} 
Let $r_\sigma=1$ if $q<\scdim g$, $r_\sigma=2$ if $q=\scdim g$, and
$\sigma^\TC=(g,H,r_\sigma)$. By construction this is a TC\--signature in $L_0$.
Given a $\lsc$\--extension $L$ of $L_0$, a tuple $(x_1,x_2)$ of
elements of $L$ is {\df SC\--primitive over $L_0$} if it is TC\--primitive
over $L_0^\TC$ and if in addition
\begin{description}
  \item[P3]
    $x_1$, $x_2$ are sc-pure of the same sc-dimension.
\end{description}
Such a SC\--primitive couple $(x_1,x_2)$ determines its so-called {\df
SC\--signature in $L_0$}, denoted by $\sigma_\SC(x_1,x_2)=(g,H,q)$
and defined as follows.
\begin{align*}
  g=g(x_1,L_0) & & H=\{g^-\meet x_1,g^-\meet x_2\} & & q=\scdim x_1
\end{align*}
Note that, by condition {\bf P2} of the definition of TC\--signatures,
$x_1=x_2$ if and only if $x_1\ll g$, and otherwise $\scdim x_1=\scdim
x_2=\scdim g$. This ensures that
$\sigma_\TC(x_1,x_2)=(\sigma_\SC(x_1,x_2))^\TC$. 

Let $L_0^\TC$ and $L^\TC$ denote the respective $\ltc$\--reducts of
$L_0$ and $L$. For every subset $X_0$ of $L$ we let:
\begin{itemize}
  \item
    $L_0\gen{X_0} =$ the $\lsc$\--structure generated by $L_0\cup X_0$ in $L$;
  \item 
    $L_0^\TC\gen{X_0}=$ the $\ltc$\--structure generated by $L_0^\TC\cup
    X_0$ in $L^\TC$.
\end{itemize}
We say that $L$ is a {\df SC-primitive extension} of $L_0$, if there exists
a tuple $(x_1,x_2)$ SC-primitive over $L_0$ such that $L=L_0\gen{x_1,x_2}$ (then
clearly $L=L_0\gen{x_1}=L_0\gen{x_2}$). By Lemma~\ref{le:SC-reduit}
below and Proposition~\ref{pr:recap-DJ}.\ref{it:recap-Irr} such a
tuple is necessarily unique. 

\begin{lemma}\label{le:SC-reduit}
  Let $L_0$ be finite subscaled lattice, and $L$ a $\lsc$\--extension%
  \footnote{See Footnote~\ref{fn:extension}}
  generated over $L_0$ by an SC\--primitive tuple $(x_1,x_2)$. Then
  $L^\TC=L_0^\TC\gen{x_1,x_2}$, $(x_1,x_2)$ is TC\--primitive over
  $L_0^\TC$ and $\sigma_\TC(x_1,x_2)=(\sigma_\SC(x_1,x_2))^\TC$. 
\end{lemma}

\begin{proof}
That $(x_1,x_2)$ is TC\--primitive over $L_0^\TC$ and
$\sigma_\TC(x_1,x_2)=(\sigma_\SC(x_1,x_2))^\TC$ is only a reminder: it follows
directly from the definitions. Let $L_1=L_0^\TC\gen{x_1,x_2}$, in
order to conclude that $L_1=L$ it only remains to prove that $L_1$ is
an $\lsc$\--substructure of $L$. By
Corollary~\ref{co:ss-lsc-structure} it suffices to check that for
every $b'<b\in\cI(L_1)$, $b$ is sc-pure in $L$ and $\scdim b'<\scdim b$
in $L$. 

If $b\in\cI(L_0)$, then $b$ is sc-pure in $L_0$ by
\ssref{SS:Ck(a)-k-pure}, hence also in $L$ because $L_0$ is an
$\lsc$\--sub\-struc\-ture of $L$. Otherwise $b=x_i$ for some
$i\in\{1,2\}$. Then $b$ is sc-pure in $L$ by definition of SC-primitive
tuples over $L_0$. 

In both cases $b'\ll b$ in $L_1$ by
\tcref{TC:a=(a-inter-b)-union-(a-b)}, that is $b-b'=b$ in $L_1$, hence
also in $L$ because $L_1$ is an $\ltc$\--substructure of $L$. So $b'\ll
b$ in $L$ hence $\scdim(b')<\scdim(b)$ in $L$ by \ssref{SS:ll-scdim}.
\end{proof}

\begin{lemma}\label{le:TC-enrichi}
  Let $L_0$ be finite subscaled lattice, $L_1$ a $\ltc$\--extension
  generated over $L_0^\TC$ by a TC\--primitive tuple $(x_1,x_2)$, and
  $\tau=(g,\{h_1,h_2\},q)$ a SC\--signature in $L_0$ such that
  $\tau^\TC=\sigma_\TC(x_1,x_2)$. Then there exists a unique
  structure of subscaled lattice expanding $L_1$ which makes
  it a $\lsc$\--extension of $L_0$ such that $(x_1,x_2)$ is
  SC\--primitive over $L_0$ and $\sigma_\SC(x_1,x_2)=\tau$. 
\end{lemma}

\begin{proof}
  By Proposition~\ref{pr:recap-DJ}.\ref{it:recap-Irr},
  $\cI(L_1)\subseteq\cI(L_0)\cup\{x_1,x_2\}$. For every $x\in\cI(L_0)$ let
  $D(x)=\scdim x$, and let $D(x_1)=D(x_2)=q$. This defines by
  restriction a function from $\cI(L_1)$ to $\NN$. Assume that $D$ is
  strictly increasing. Then it determines as in
  Example~\ref{ex:Lnoeth} an $\lsc$\--structure on $L_1$ expanding its
  $\ltc$\--structure. Let us denote it $L$, so that $L^\TC=L_1$. Every
  $\join$\--irreducible element of $L_0$ remains sc-pure in $L$ with the
  same sc-dimension, hence by
  Proposition~\ref{pr:CNS-de-lsc-plongement} the inclusion of $L_0$
  into $L$ is an $\lsc$\--embedding. This is clearly the only
  possible $\lsc$\--structure on $L_1$ which makes it an
  $\lsc$\--extension of $L$ such that $\scdim x_1=\scdim x_2=q$. So it
  only remains to prove that $D$ is strictly increasing. 
  
  Let $b< a$ in $\cI(L)$, if $a,b\in\cI(L_0)$ then
  $D(b)<D(a)$ by \ssref{SS:ll-scdim}. So we can assume that $a$ or $b$
  does not belong to $\cI(L_0)$. By
  Proposition~\ref{pr:recap-DJ}.\ref{it:recap-Irr} one of them must
  belong to $\{x_1,x_2\}$ and the other one to $\cI(L_0)$. Note that our
  assumption $\sigma_\TC(x_1,x_2)=\sigma^\TC$ implies that (for $i=1,2$)
  $g=g(x_i,L_0)$, and $h_i=x_i\meet g^-$ (up to re-numbering) and: either
  $x_1=x_2$, $h_1=h_2<g$ and $\scdim h_1<q<\scdim g$; or $x_1\neq x_2$,
  $h_1\join h_2=g^-$ and $q=\scdim g$. 

  {\it Case 1}: $b=x_1$ or $b=x_2$, hence $D(b)=q$. Then $a\in\cI(L_0)$,
  in particular $a\in L_0$, hence $g\leq a$ and so $\scdim g\leq \scdim a$. If
  $x_1=x_2$ then $q<\scdim g\leq \scdim a$ hence $D(b)<D(a)$. If $x_1\neq x_2$ 
  then $q=\scdim g$, and $g$ is not
  $\join$\--irreducible in $L$ hence $g\neq a$. So $g\ll a$ (because $g<a$ and
  $a$ is $\join$\--irreducible) hence $\scdim g<\scdim a$ by
  \ssref{SS:ll-scdim}, that is $D(b)<D(a)$. 

  {\it Case 2}: $a=x_1$ or $a=x_2$, hence $D(a)=q$. Then again 
  $b\in L_0$, and $b<a\leq g$ hence $b\leq g^-$. If $x_1=x_2$, since $b\leq a\meet
  g^-=h_1$ we get $\scdim b\leq\scdim h_1<q$, hence $D(b)<D(a)$. If $x_1\neq
  x_2$ then $\scdim g=q$. Since $b<g$ we have $b\ll g$ (because $g$
  is $\join$\--irreducible) hence $\scdim b<\scdim g$ by
  \ssref{SS:ll-scdim}. So $\scdim b<q$, that is $D(b)<D(a)$.
\end{proof}

We can now pack all this together. We will refer to the $k$\--th item
of the above proposition as Proposition~\ref{pr:SC-irr-sign-min}.$k$. 

\begin{proposition}\label{pr:SC-irr-sign-min}
  Let $L_0$ be a finite subscaled lattice and $L$ a
  $\lsc$\--extension%
  \footnote{See Footnote~\ref{fn:extension}.}. 
  \begin{enumerate}
    \item\label{it:SC-Irr}
      If $L$ is $\lsc$\--generated over $L_0$ by a SC\--primitive
      tuple $(x_1,x_2)$, then $L$ is exactly the upper semi-lattice
      generated over $L_0$ by $x_1$ and $x_2$. It is a finite
      subscaled lattice and one of the following holds: 
      \begin{enumerate}
        \item
          $x_1=x_2$ and $\cI(L)=\cI(L_0)\cup\{x_1\}$.
        \item
          $x_1\neq x_2$ and $\cI(L)=\left(\cI(L_0)\setminus\{g\}\right)\cup\{x_1,x_2\}$.
      \end{enumerate}
    \item\label{it:SC-sgn-ext}
      SC\--signatures in $L_0$ and SC\--primitive extensions of $L_0$
      are in one-to-one correspondence: every SC\--signature in $L_0$
      is the SC\--signature of a SC\--primitive extension, and two
      SC\--primitive extensions of $L_0$ are $\lsc$\--isomorphic over
      $L_0$ if and only if they have the same SC\--signature in $L_0$. 
    \item\label{it:SC-prim-min}
      If $L$ is finite, the following are equivalent. 
      \begin{enumerate}
        \item\label{it:SC-min}
          $L$ is a minimal proper $\lsc$\--extension of $L_0$.
        \item\label{it:SC-prim}
          $L$ is a SC\--primitive extension of $L_0$.
        \item\label{it:SC-card}
          $\card(\cI(L))=\card(\cI(L_0))+1$.
      \end{enumerate}
      As a consequence every finite $\lsc$\--extension $L'$ of $L_0$ is the
      union of a tower of SC\--primitive extensions $L_0\subset L_1\subset\cdots\subset
      L_n=L'$ with $n=\card(\cI(L'))-\card(\cI(L_0))$. 
  \end{enumerate}
\end{proposition}

If $L$ is a SC\--primitive extension of a finite subscaled lattice
$L_0$, by Proposition~\ref{pr:SC-irr-sign-min}.\ref{it:SC-Irr} it is
generated over $L_0$ by a {\em unique} (up to permutation)
SC-primitive couple $(x_1,x_2)$. We call $\sigma_\SC(x_1,x_2)$ the {\df
SC\--signature of $L$ in $L_0$} and denote it $\sigma_\SC(L)$.

\begin{proof}
(\ref{it:SC-Irr}) If $L$ is $\lsc$\--generated over $L_0$ by an SC\--primitive
tuple $(x_1,x_2)$, then by Lemma~\ref{le:SC-reduit} $L^\TC$ is also
$\ltc$\--generated over $L_0^\TC$ by $(x_1,x_2)$, which is
TC\--primitive. The first item the follows from
Proposition~\ref{pr:recap-DJ}.\ref{it:recap-Irr}. 

(\ref{it:SC-sgn-ext}) Let $\sigma$ be an SC\--signature in $L_0$. Then
$\sigma^\TC$ is a TC\--signature in $L_0$.
Proposition~\ref{pr:recap-DJ}.\ref{it:recap-sgn-ext} gives a
TC\--primitive $\ltc$\--extension $L_1$ of $L_0^\TC$ with
TC\--signature $\sigma^\TC$ in $L_0$. Lemma~\ref{le:TC-enrichi} then gives
a unique structure of subscaled lattice expanding $L_1$ which makes it
an SC\--primitive extension of $L_0$ with signature $\sigma$ in $L_0$. Let
us denote it $L$, so that $L^\TC=L_1$. Now if $L'$ is another
SC\--primitive extension with signature $\sigma$ in $L_0$, by
Proposition~\ref{pr:recap-DJ}.\ref{it:recap-sgn-ext} $L'^\TC$ is
$\ltc$\--isomorphic to $L^\TC$ over $L_0$. The image of $L'$ {\it via}
this endomorphism defines an $\lsc$\--structure expanding $L^\TC$,
which makes it an SC\--primitive extension of $L_0$ with the same
signature as $L$. By the uniqueness of such a structure, given by
Lemma~\ref{le:TC-enrichi}, it follows that this $\ltc$\--isomorphism
from $L'^\TC$ to $L^\TC$ is actually an $\lsc$\--isomorphism, which
proves the result. 

(\ref{it:SC-prim-min}) We prove
(\ref{it:SC-min})$\Rightarrow$(\ref{it:SC-prim})$\Rightarrow$(\ref{it:SC-card})$\Rightarrow$(\ref{it:SC-min}).
Note that (\ref{it:SC-prim})$\Rightarrow$(\ref{it:SC-card}) follows from
item~\ref{it:SC-Irr}). 

(\ref{it:SC-card})$\Rightarrow$(\ref{it:SC-min}). Let $L'$ be a proper
$\lsc$\--extension of $L_0$ contained in $L$. Then $L'^\TC$ is a
proper $\ltc$\--extension of $L_0^\TC$ contained in $L^\TC$. By
Proposition~\ref{pr:recap-DJ}.\ref{it:recap-prim-min},
(\ref{it:SC-card}) implies that $L^\TC$ is a minimal proper
$\ltc$\--extension of $L_0^\TC$. So $L'^\TC=L^\TC$, thus necessarily
$L'=L$, which proves that $L$ is minimal. 

(\ref{it:SC-min})$\Rightarrow$(\ref{it:SC-prim}). Let $x_1$ be a minimal
element in $\cI(L)\setminus\cI(L_0)$. Let $g=g(x_1,L_0)$,; if $x_1\ll g$ let
$x_2=x_1$, otherwise let $x_2=g-x_1$. The proof of Corollary~3.4 in
\cite{darn-junk-2018} shows that $(x_1,x_2)$ is TC\--primitive
over $L_0^\TC$. In particular $x_1,x_2\in\cI(L)$ so they are
sc\--pure by \ssref{SS:Ck(a)-k-pure}. The same holds true for $g$,
hence if $x_1\neq x_2$ then $x_1=g-x_2$ and $x_2=g-x_1$ have the same
dimension (the dimension of $g$, by definition of the sc-purity of
$g$). So $(x_1,x_2)$ is actually SC\--primitive. Since
$L_0^\TC\gen{x_1,x_2}=L^\TC$, {\it a fortiori} $L_0\gen{x_1,x_2}=L$,
hence $L$ is SC\--primitive over $L_0$. 
\end{proof}

\section{Model-completion of scaled lattices}
\label{se:model-completion}

We say that a subscaled lattice $L$ is a {\df super scaled lattice},
if $L$ satisfies the following additional properties, both of which are clearly
axiomatizable by $\forall\exists$\--formulas in $\lsc$. Moreover, if $\scdim L\leq
d$, we say that $L$ is a {\df super $d$\--scaled lattice}.
\begin{description}
  \item[\bf Catenarity]
    For every non-negative integers $r\leq q\leq p$ and every elements $c\leq
    a\neq\ZERO$, if $c$ is $r$\--sc-pure and $a$ is $p$\--sc-pure then
    there exists a non-zero $q$-sc-pure element $b$ such that $c\leq b\leq
    a$.
\end{description}

If $\Sp L$ is noetherian this property is equivalent to the usual
notion of catenarity, namely that any two maximal chains in $\Sp L$
having the same first and last elements have the same length. In
particular every $d$\--scaled lattice $L$ of type $\lzar(X)$ or
$\llin(X)$ satisfies this property. If $\cK$ is an $o$\--minimal
field and $X\subseteq K^m$ is any definable set, then $\ldef(X)$ also
satisfies the Catenarity Property: given $A,C\in\ldef(X)$, respectively
$p$\--pure and $r$\--pure, the Triangulation Theorem reduces to the
case where $A$ is a simplex and $C$ one of its faces, and it then
suffices to take for $B$ a face of $A$ of dimension $p$ containing $C$
In contrast, none of these scaled lattices satisfy the next property,
as it implies that $L$ is atomless. 

\begin{description}
  \item[\bf Splitting]
    For every elements $b_1,b_2,a$, if $b_1\join b_2\ll a\neq\ZERO$ then there
    exists non-zero elements $a_1\geq b_1$ and $a_2\geq b_2$ such that:
    \[
      \left\{
        \begin{array}{l}
          a_1=a-a_2\\
          a_2=a-a_1\\
          a_1 \meet a_2 = b_1\meet b_2 
        \end{array}
      \right. 
    \]
    We will then say $a_1$, $a_2$ {\df split
    $a$ along} $b_1$, $b_2$. 
\end{description}

\begin{remark}\label{re:cat-0}
  If $r< p\leq q$ in the Catenarity axiom, the conclusion can be
  strengthen to $c\ll b\leq a$. Indeed $b$ has pure sc-dimension $q$ 
  and $c\meet b=c$ has sc-dimension $<q$ hence $b-c=b$ by
  \ssref{SS:Ck-a-moins-b}. In particular {\em every
    subscaled lattice satisfying the Catenarity axiom is a scaled
  lattice}. Indeed, given any element $a$ of sc-dimension
  $d\geq1$, repeated applications of the Catenarity axiom to $\pc^d(a)$,
  $c=\ZERO$ and each integer $p$ from $0$ to $d$, gives a chain of
  sc-pure elements $a_0,\dots,a_d$ such that
  \[
    \ZERO\neq a_0\ll a_1\ll \cdots \ll a_d\leq a. 
  \]
  By Fact~\ref{fa:ldim-et-TCdim} it follows that $\dim a\geq d$, and
  by \ssref{SS:scdim-inf-dim} that $\dim a=d$. 
\end{remark}

\begin{lemma}\label{le:embed-split}
  Let $a,b_1,b_2$ be elements of a finite subscaled lattice $L_0$. If
  $b_1 \join b_2 \ll a \neq \ZERO$ then $L_0$ embeds in a finite subscaled lattice
  $L$ containing non-zero elements $a_1$, $a_2$ which split $a$ along
  $b_1$, $b_2$. Moreover, if $\pc^0(a)=\ZERO$ we can
  require\footnote{This additional requirement when $\pc^0(a)=\ZERO$ will
    be used only later, in Section~\ref{se:mod-comp-ASC}.} that all
  the atoms of $L$ belong to $L_0$. 
\end{lemma}

\begin{proof}
  We are going to prove by induction on $d=\scdim a$ a slightly more
  precise result, namely that in addition $x\leq a$ for every
  $x\in\cI(L)\setminus\cI(L_0)$. 
  Let $g_1,\dots,g_n$ be the $\join$\--irreducible components of $a$ in $L$.
  Note that $n\geq1$ because $a\neq\ZERO$. If $d=0$ our assumption that
  $b_1\join b_2\ll a$ implies by \ssref{SS:ll-scdim} that $b_1=b_2=\ZERO$.
  If $n=1$, that is $a=g_1$ is $\join$\--irreducible, then
  $\sigma=(g,\{\ZERO\},0)$ is a signature in $L$.
  Proposition~\ref{pr:SC-irr-sign-min}.\ref{it:SC-sgn-ext} gives an
  SC\--primitive couple $(a_1,a_2)$ generating an $\lsc$\--extension
  $L_1$ over $L$ with signature $\sigma$. This signature ensures that
  $(a_1,a_2)$ splits $a$ along $(\ZERO,\ZERO)$. If $n\geq2$,
  $a_1=g_1$ and $a_2=a-a_1$ will do the job. So the result is proved 
  for $d=0$. 
  
  Now assume that $d\geq1$ and the result is valid until $d-1$. Note
  that $g_1^-\join\cdots\join g_n^-$ is the greatest element $c\in L$ such that
  $c\ll a$, in particular
  \begin{equation}
    b_1\join b_2\leq g_1^-\join\cdots\join g_n^-.
    \label{eq:b1b2-g-moins}
  \end{equation}
  Let $u=(\jjoin_{i\leq n}g_i^-)-(b_1\join b_2)$ and $u^*=u-\pc^0(u)$. Since $u\ll
  a$ we have $\scdim u< d$ by \ssref{SS:ll-scdim}. 
  
  We are claiming that $L_0$ embeds in a finite subscaled lattice $L$
  without new atoms, in which all the $g_i$'s are still
  $\join$\--irreducible with the same predecessor as in $L_0$, and in
  which there are elements $u^*_1$, $u^*_2$ which satisfy all the
  conditions to split $u^*$ along $b_1\meet u^*$, $b_2\meet u^*$, except that
  $u^*_1$, $u^*_2$ might be zero elements.
  
  By \tcref{TC:(x-y)-y}, $(b_1\join b_2)\meet u^*\ll u^*$ so if $\scdim u\leq 0$ we
  can simply take $u_1^*=u_2^*=\ZERO$ and $L_0=L$. On the other hand,
  if $\scdim u>0$ the induction hypothesis applies to $u^*$, $b_1\meet
  u^*$, $b_2\meet u^*$. It gives a finite subscaled lattice $L$
  containing $L_0$ and elements $u^*_1,u^*_2\in L_0$ which split $u^*$
  along $b_1\meet u^*$, $b_2\meet u^*$. Moreover we can require that $L$ do
  not contain any new atom because $\pc^0(u^*)=\ZERO$, and that $x\leq
  u^*$ for every $x\in\cI(L)\setminus\cI(L_0)$. For every $x\in\cI(L)$ such that
  $x<g_i$ for some $i\leq n$, if $x\in L_0$ then $x\leq g_i^-$ (where $g_i^-$
  still denotes the predecessor of $g_i$ in $L_0$). If $x\notin L_0$ then
  $x\leq u^*$ by construction hence $x\leq g_i\meet u^*$. The latter belongs to
  $L_0$ and is strictly smaller than $g_i$, hence smaller than $g_i^-$, so
  $x<g_i$. It follows that $g_i^-$ is still the unique predecessor of
  $g_i$ in $L$. In particular $g_i$ remains $\join$\--irreducible in
  $L$. This proves our claim in both cases. 
  
  Now let $u_1=\pc^0(u)\join u_1^*$ and $u_2=u_2^*$. We have in particular
  \begin{equation}
    u_1\join u_2=\jjoin_{i\leq n}g_i^- -(b_1\join b_2).
    \label{eq:u1-join-u2}
  \end{equation}
  Since $u-b_2=u$ by \tcref{TC:(x-y)-y} necessarily $b_2\meet c\ll c$ for
  every $\join$\--irreducible component $c$ of $u$, hence
  $b_2\meet\pc^0(u)\ll\pc^0(u)$. By \ssref{SS:ll-scdim} it follows that
  $b_2\meet\pc^0(u)=\ZERO$ hence $b_2\meet u_1=b_2\meet u_1^*$. Similarly
  $u^*\meet\pc^0(u)=\ZERO$ because $u^*-\pc^0(u)=u^*$ by
  \ssref{SS:a-moins-Cd(a)} and \tcref{TC:(x-y)-y}. {\it A fortiori}
  $u_2^*\meet\pc^0(u)=\ZERO$ hence $u_2^*\meet u_1=u_2^*\meet u_1^*$. Note also
  that $b_1\meet u_2^*=b_1\meet u^*\meet u_2^*\leq u_1^*\meet u_2^*$, and symmetrically
  $b_2\meet u_1^*\leq u_1^*\meet u_2^*$. Altogether, since $u_2=u_2^*$ and
  $u^*_1\meet u^*_2\leq b_1\meet b_2$ by construction, this gives
  \[
    (b_1\meet u_2)\join(b_2\meet u_1)\join(u_1\meet u_2)\leq(b_1\meet b_2) 
  \]
  hence
  \begin{equation}
    (b_1\join u_1)\meet(b_2\join u_2)
    =(b_1\meet b_2)\join(b_1\meet u_2)\join(b_2\meet u_1)\join(u_1\meet u_2)=(b_1\meet b_2).
    \label{eq:b1u1-inter-b2u2}
  \end{equation}

  After this preparation, for each $i$ let 
  \begin{align*}
    h_{i,1}=g_i^-\meet (b_1\join u_1), & & h_{i,2}=g_i^-\meet (b_2\join u_2), 
    & & \sigma_i=\big(g_i,\{h_{i,1},h_{i,2}\},\scdim g_i\big)
  \end{align*}
  Using (\ref{eq:u1-join-u2}) we get
  \begin{align*}
    h_{i,1}\join h_{i,2}
    &= g_i^-\meet (b_1\join u_1\join b_2\join u_2 ) \\
    &= g_i^-\meet\left[b_1\join b_2\join \left(\jjoin_{j\leq n}g_j^- -(b_1\join b_2)\right)\right]\\
    &= g_i^-\meet \jjoin_{j\leq n}g_j^- =g_i^-.
  \end{align*}
  So each $\sigma_i$ is an SC\--signature in $L_0$. In particular
  Proposition~\ref{pr:SC-irr-sign-min}.\ref{it:SC-sgn-ext} gives an
  SC\--primitive extension $L_1=L_0\gen{a_{1,1},a_{1,2}}$ with
  SC\--signature $\sigma_1$ in $L_0$. By
  Proposition~\ref{pr:SC-irr-sign-min}.\ref{it:SC-Irr},
  $\cI(L_1)=(\cI(L_0)\setminus\{g_1\})\cup\{a_{1,1},a_{1,2}\}$. In particular
  $g_2\in\cI(L_1)$, hence $\sigma_2$ is still an SC\--signature in $L_1$.
  Repeating the construction $n$ times (note that $a\neq\ZERO$ ensures
  that $n\geq1$) gives a chain of $\lsc$\--extensions $(L_i)_{i\leq n}$ and
  for each $i>0$, an SC\--primitive couple $(a_{i,1},a_{i,2})$
  generating $L_i$ over $L_{i-1}$ with signature $\sigma_i$ in $L_{i-1}$.
  Each $g_i=a_{i,1}\join a_{i,2}$ and by
  Proposition~\ref{pr:SC-irr-sign-min}.\ref{it:SC-Irr}
  \begin{equation}
  \cI(L_n)=\big(\cI(L_0)\setminus\{g_1,\dots,g_n\}\big)
           \cup\{a_{1,1},a_{1,2},\dots,a_{n,1},a_{n,2}\}
    \label{eq:irr-Ln}
  \end{equation}
  so $a_{1,1},a_{1,2},\dots,a_{n,1},a_{n,2}$ are the $\join$\--irreducible
  components of $a$ in $L_n$. Moreover every $c\in\cI(L_n)$ such that
  $c<a_{i,k}$ for some $i,k$ must belong to $L_0$, hence the
  predecessor of $a_{i,k}$ is the same in every $L_j$ and belongs to
  $L_0$. We can then denote it $a_{i,k}^-$ without ambiguity, and by
  construction we have 
  \begin{equation}
    a_{i,k}^- = a_{i,k}\meet g(a_i,L_{i-1}) = a_{i,k}\meet g_i = h_{i,k}.
    \label{eq:aik-moins}
  \end{equation}
  Let $a_1=\jjoin_{i\leq n}a_{i,1}$, $a_2=\jjoin_{i\leq n}a_{i,2}$, $h_1=\jjoin_{i\leq
  n}h_{i,1}$ and $h_2=\jjoin_{i\leq n}h_{i,2}$. We are going to check that
  $a_1$, $a_2$ split $a$ along $b_1$, $b_2$. Both of them are non-zero
  and since the $a_{i,k}$'s are the $\join$\--irreducible components of
  $a$ we have $a-a_1=a_2$, $a-a_2=a_1$. Each
  $a_{i,1}\geq h_{i,1}$ by construction, hence $a_1\geq h_1$ and
  symmetrically $a_2\geq h_2$. Moreover for $k\in\{1,2\}$
  \[
    h_k=\jjoin_{i\leq n} h_{i,k}\geq\jjoin_{i\leq n}g_i^-\meet b_k=b_k 
  \]
  where the last equality comes from (\ref{eq:b1b2-g-moins}), so
  $a_k\geq b_k$. It remains to check that $a_1\meet a_2=b_1\meet b_2$. 

  For $i\neq j$, $a_{i,1}$ and $a_{j,2}$ are mutually incomparable
  hence by (\ref{eq:aik-moins})
  \[
    a_{i,1}\meet a_{j,2}=a_{i,1}^-\meet a_{i,2}^- =h_{i,1}\meet h_{j,2}.
  \]
  On the other hand $a_{i,1}\meet a_{i,2}=h_{i,1}\meet h_{i,2}$ by
  construction. The conclusion follows, with $L=L_n$, using
  (\ref{eq:b1u1-inter-b2u2}).
  \begin{align*}
    a_1\meet a_2 &= \jjoin_{i,j}a_{i,1}\meet a_{j,2} = \jjoin_{i,j}h_{i,1}\meet h_{j,2} \\
             &= \jjoin_{i,j}\big[g_i^-\meet (b_1\join u_1)\big]
                      \meet \big[g_j^-\meet (b_2\join u_2)\big] \\
             &= \jjoin_{i,j}(g_i^-\meet g_j^-)
                      \meet\big[(b_1\join u_1)\meet(b_2\join u_2)\big] \\
             &= \Big(\jjoin_{i}g_i^-\Big)\meet \Big(\jjoin_{j}g_j^-\Big)
                      \meet\big[(b_1\join u_1)\meet(b_2\join u_2)\big] \\
             &= (b_1\join u_1)\meet(b_2\join u_2) = b_1\meet b_2.
  \end{align*}
\end{proof}

\begin{theorem}\label{th:model-completion}
The theory of super $d$\--scaled lattices is the 
model-completion of the theory of $d$\--subscaled lattices. In
particular, it eliminates the quantifiers in $\lsc$.
\end{theorem}

\begin{proof} 
  The last statement follows from the first one, as is usual for the
  model-completion of a universal theory. 
  By standard model-theoretic arguments it then suffices to prove that
  every existentially closed $d$\--subscaled lattice is super
  $d$\--scaled, and that for every super $d$\--scaled lattice $\hat
  L$, every finitely generated $d$\--subscaled lattice $L$ and every
  common $\lsc$\--substructure $L_0$, there is an embedding of $L$
  into $\hat L$ over $L_0$. 

  Let $L$ be an existentially closed $d$\--subscaled lattice, and
  $L_0$ a finitely generated substructure. By
  Theorem~\ref{th:TCS-treillis-type-fini}, $L_0$ is finite. By
  Proposition~\ref{pr:linear-repres}, $L_0$ $\lsc$\--embeds the
  $d$\--scaled lattice $\ldef(X)$ of some special linear set $X$,
  which is in particular a Catenary lattice. By the model-theoretic
  compactness Theorem it follows that $L$ is catenary. Similarly
  Theorem~\ref{th:TCS-treillis-type-fini}, Lemma~\ref{le:embed-split}
  and the model-theoretic compactness Theorem prove that $L$ has the
  Splitting property, hence $L$ is super $d$\--scaled. 

  Conversely assume that $\hat L$ is a super $d$\--scaled lattice, $L$ a
  finitely generated $d$\--subscaled lattice, and $L_0$ is a
  common $\lsc$\--substructure of both. By
  Theorem~\ref{th:TCS-treillis-type-fini} and
  Proposition~\ref{pr:SC-irr-sign-min}.\ref{it:SC-min}, we are reduced
  to the case where $L$ is a primitive extension of $L_0$. 
  Let $\sigma=(g,\{h_1,h_2\},q)$ be its SC\--signature. By
  Proposition~\ref{pr:SC-irr-sign-min}.\ref{it:SC-sgn-ext} it suffices
  to find a $x_1,x_2\in\hat L$ such that $(x_1,x_2)$ is SC\--primitive
  over $L_0$ and $\sigma_\SC(x_1,x_2)=\sigma$. We distinguish two cases, and let $g^-$
  denotes the predecessor of $g$ in $L_0$.

{\it Case 1}: $\scdim h_1< q<\scdim g$ and $h_1=h_2<g$. Let $p=\scdim
g$ and $r=\scdim h_1$. Let $y_1,y_2\in \hat L$ which split $g$ along
$h_1,g^-$. For $0\leq i\leq d$, either $i<q$ or $\pc^i(h_1)=\ZERO$ (because
$\scdim h_1=r<q$), hence $\scdim \pc^i(h_1)<q<\scdim g$. Recall that
$g$ is sc\--pure and $y_1=g-y_2\neq\ZERO$, so and $y_1$ has pure
sc\--dimension $p$ like $g$. The Catenarity property then applies to
$\pc^i(h_1)\leq y_1=\pc^p(y_1)$ and gives $x_i\in\hat L$ such that
$\pc^i(h_1)\leq x_i \leq y_1$ and $x_i$ has pure sc\--dimension $q$. Let
$x=\jjoin_{0\leq i\leq d}x_i$, by construction $h_1=\jjoin_{i\leq d}\pc^i(h_1)\leq x\leq y_1$
and $x$ has pure sc\--dimension $q$. In particular 
\[
  h_1\leq x\meet g^-\leq y_1\meet y_2 =h_1\meet g^-=h_1 
\]
hence $x\meet g^-=h_1\in L_0$. Moreover $x\meet g^-=h_1\ll x$ because $\scdim h_1<q$
and $x$ has pure sc\--dimension $q$. Finally $x\ll g$ because $\dim x_i=q<p$ and
$g$ has pure sc\--dimension $p$. Altogether this proves that $(x,x)$ is an
SC\--primitive tuple over $L_0$ with SC\--signature $\sigma$. 

{\it Case 2}: $q=\scdim g$ and $h_1\join h_2=g^-$. Let $y_1,y_2\in \hat L$
which split $g$ along $h_1,h_2$. By
construction $y_1\join y_2= g$, and since $g$ has pure sc\--dimension $q$
so does each $y_i$. In addition $y_1\meet y_2=h_1\meet h_2\in L_0$. Moreover
\[
  y_1\meet h_2\leq y_1\meet y_2=h_1\meet h_2 
\]
hence $y_1\meet (h_1\join h_2)=h_1\join(y_1\meet h_2)=h_1$. Since $h_1\join h_2=g^-$ it
follows that $y_1\meet g^-=h_1\in L_0$, and symmetrically $y_2\meet g^-=h_2\in
L_0$. So $(y_1,y_2)$ is an SC\--primitive tuple over $L_0$ with SC\--signature
$\sigma$. 
\end{proof}

\begin{remark}\label{re:sign-dans-hat-L}
  The proof of Theorem~\ref{th:model-completion} shows that if $L_0$
  is a finite $\lsc$\--substructure of a super scaled lattice $\hat
  L$, then every signature $\sigma$ in $L_0$ is the signature of an
  SC\--primitive extension of $L_0$ in $\hat L$. 
\end{remark}

The completions of the theory of super $d$\--scaled lattices 
are easy to classify. Let us say that a $d$\--subscaled lattice 
is {\df prime} if it does not contain any proper $d$\--subscaled 
lattice, or equivalently if it is generated by the empty set. Every 
prime $d$\--subscaled lattice is finite. By 
Corollary~\ref{cor:structures-a-n-generateurs} there exists 
finitely many prime $d$\--subscaled lattices up to isomorphism. 

\begin{corollary}\label{cor:completions}
The theory of super $d$\--scaled lattices containing (a copy of) 
a given prime $d$\--subscaled lattice is $\aleph_0$\--categorical, 
hence complete. It is also recursively axiomatizable, hence decidable. 
Every completion of the theory of super $d$\--scaled lattices
is of that kind, and the theory of super $d$\--scaled lattices
is decidable. 
\end{corollary}

\begin{proof}
  Let $L$, $L'$ be any two countable super $d$\--scaled lattices
  containing isomorphic prime $d$\--subscaled lattices $L_0$ and
  $L'_0$. By Remark~\ref{re:sign-dans-hat-L} any partial isomorphism
  between $L$ and $L'$, extending the given isomorphism between $L_0$
  and $L'_0$, can be extended by a back and forth process. This proves the first
  statement. The other ones are immediate consequences. 
\end{proof}

\section{Atomic scaled lattices}
\label{se:atom-scal-latt}

Every super scaled lattice is atomless because of the Splitting
Property, hence none of the geometric scaled lattice amongst
$\sdef(\cK,d)$, $\szar(K,d)$, $\slin(K,d)$ can be super scaled. In
order to apply our study to some of them, we now introduce a variant of
subscaled lattices intended to protect atoms against splitting. 
\\

Let $\lasc=\lsc\cup\{\AT_k\}_{k\in\NN^*}$, with each $\AT_k$ a new unary
predicate symbol. For any $\lasc$\--structure $L$ we denote by
$\AT_k(L)$ the set of elements $a$ in $L$ such that $L\models\AT_k(a)$, and
we let $\AT_0(L)=L\setminus\bigcup_{k>0}\AT_k(L)$. We call $L$ an {\df
ASC-lattice} if its $\lsc$\--reduct is a scaled lattice and if
it satisfies the following condition.
\newcommand{\assref}[1]{$\rm ASC_{\ref{#1}}$}
\begin{list}{$\rm\bf ASC_{\theenumi}$:}{\usecounter{enumi}}
\setcounter{enumi}{-1}
\item
\label{ASS:def}
$(\forall k>0)$, $a\in\AT_k(L)$ if and only if $a$ is the join
of exactly $k$ atoms in $L$. 
\end{list}

\begin{remark}\label{re:ASC-axiom-AE}
  This condition can be expressed by $\forall\exists$ formulas in $\lasc$ by
  saying first that $\AT_1(L)$ is the set of atoms of $L$, and then
  that $\AT_k(L)$ is the set of elements of $L$ which are the join of
  exactly $k$ elements of $\AT_1(L)$.
\end{remark}

Every ASC-lattice obviously satisfies also the following
schemes (for $k,l>0$) of universal axioms:
\begin{list}{$\rm\bf ASC_{\theenumi}$:}{\usecounter{enumi}}
  \item\label{ASS:Atk-disjoints}
    $(\forall k,l>0,\ k\neq l),\quad \forall a,\ \AT_k(a)\to\lnot\AT_l(a)$
  \item\label{ASS:Atk-adb}
    $(\forall k>0),\quad \forall a,a_0,\dots,a_{2^k},\quad \AT_k(a)\;\longrightarrow$ 

    \centerline{
      $\displaystyle
        \left[
        \cconj_{0\leq i\leq 2^k}\bigl(a_i\leq a\bigr)\longrightarrow
        \ddisj_{0\leq i<j\leq 2^k}\bigl(a_i=a_j\bigr)
        \right] \bigwedge \scdim a= 0
      $
    }
  \item\label{ASS:Atk-atomes}
    $(\forall k>0)$,\quad $\forall a,a_1,a_2$, 

    \centerline{
      $\displaystyle
        \Bigl[\big(a=a_1\join a_2\big) \conj \big(a_1\meet a_2=\ZERO\big) \conj 
        \big(a_1\neq\ZERO\big) \conj \big(a_2\neq\ZERO\big)\Bigr]
      $
    }

    \centerline{
      $\displaystyle \longrightarrow\;
        \Biggl[\AT_k(a) \longleftrightarrow 
          \ddisj_{0< l< k} \big(\AT_l(a_1) \conj \AT_{k-l}(a_2)\big)
        \Biggr]
      $
    }
 \end{list}
We call {\df sub-ASC-lattices} the $\lasc$\--structures $L$ whose
$\lsc$\--reduct is a subscaled lattice and which satisfy
\assref{ASS:Atk-disjoints} to \assref{ASS:Atk-atomes} (but not
necessarily \assref{ASS:def}). 

The scheme \assref{ASS:Atk-disjoints} obviously means that
$(\AT_k(L))_{k\in\NN}$ is a partition\footnote{By a
  ``partition'' a set $S$, we mean here a collection of disjoint sets
  $X$ covering $S$. In particular, we do not require these sets $X$ to
  be non-empty.\label{fn:partition}} of $L$. For any $a\in
L$ we then define $\asc(a)$ as the unique $k\in\NN$ such that
$a\in\AT_k(L)$. 

The scheme \assref{ASS:Atk-adb} says that if $\asc(a)=k>0$ then $L(a)$ has at
most $2^k$ element and $\scdim(a)=0$. Then $\dim(a)=0$ by
\ssref{SS:scdim-inf-dim} so $L(a)$ is a co-Heyting algebra
with dimension $0$, hence a Boolean algebra. So \assref{ASS:Atk-adb}
actually says that $\scdim a=0$ and $L(a)$ is a Boolean algebra with
$n$ atoms for some non zero $n\leq k$. In particular every $a\in\AT_1(L)$
is an atom of $L$. 

The scheme \assref{ASS:Atk-atomes} says that if $a$ is the join of two
non-zero disjoint elements $a_1$, $a_2$ then $\asc(a)$ is non-zero if and only if
$\asc(a_1)$ and $\asc(a_2)$ are non-zero, in which case
$\asc(a)=\asc(a_1)+\asc(a_2)$. By a straightforward induction this
extends to any decomposition of $a$ as the join of finitely many
pairwise disjoint elements. In view of \assref{ASS:Atk-adb} it then
says that $\asc(a)>0$ if and only if $a$ is the join of finitely many
atoms $a_1,\dots,a_n$ of $L$ such that each $\asc(a_i)>0$, in which case
$\asc(a)=\sum_{1\leq i\leq n}\asc(a_i)$.

\begin{remark}\label{re:lasc-plongement}
  It follows immediately that a $\lsc$\--embedding of
  sub-ASC-lattices $\varphi\colon L\to L'$ is an $\lasc$\--embedding if and only
  if $\asc(a)=\asc(\varphi(a))$ for every {\it atom} $a\in L$.
\end{remark}

\begin{remark}\label{re:ASC-loc-fini}
  Obviously every finitely generated substructure of a
  sub-ASC-lattices is finite by the Local Finiteness
  Theorem~\ref{th:TCS-treillis-type-fini}, because $\lasc$ expands
  $\lsc$ only by relational symbols. 
\end{remark}

Every scaled lattice $L$ admits a unique structure of ASC-lattice
which is an expansion by definition of its lattice structure. We
denote by $L^{\rm At}$ this expansion of $L$. 

\begin{proposition}[Linear representation]\label{pr:repres-lin-ASC}
Let $K$ be an infinite field and $L_0$ be a finite sub-ASC-lattice.
For every integer $N\geq0$ there exists a special linear set $X_N$ over
$K$ and a $\lasc^*$\--embedding $\varphi_N\colon L_0\to\lalin(K^m)$ such that for
every atom $a$ of $L_0$ we have:
\begin{itemize}
  \item
    If $\asc(a)>0$ then $\asc(\varphi_N(a))=\asc(a)$.
  \item 
    If $\asc(a)=0$ then $\varphi_N(a)$ is greater than at least $N$ atoms. %$\asc(\varphi_N(a))\geq N$.
\end{itemize}
\end{proposition}

\begin{proof}
By induction on lexicographically ordered tuples of integers $(r,s)$ 
we prove that the result is true for every finite sub-ASC-lattice
$L_0$ having $r$ $\join$\--irreducible elements, $s$ of which
have the same sc\--dimension as $L_0$. 

If $r=0$ then $s=0$ and the unique embedding of $L_0=\{\ZERO\}$ into
$\lalin(P)$, for an arbitrary point $P$ of $K$, has the required
property. So let us assume that $r\geq 1$ and that the result is proved
for every $(r',s')<(r,s)$. Let $d=\scdim L_0$ and $a_1,\dots,a_r$ be the
elements of $\cI(L_0)$ ordered by increasing sc\--dimension, so that
$\scdim a_r=d\geq 0$.

{\it Case~1}: $d=0$.
Then $L_0$ is a boolean algebra and $a_1,\dots,a_r$ are its atoms. 
Let $A_1,\dots,A_r$ be pairwise disjoint subsets of $K$ such that:
\begin{itemize}
  \item
    If $\asc(a_i)>0$ then $A_i$ has $\asc(a_i)$ elements, so
    $\asc(A_i)=\asc(a_i)$.
  \item
    If $\asc(a_i)=0$ then $A_i$ has $N$ elements, so $\asc(A_i)=N$. 
\end{itemize}
Let $X$ be the union of all these $A_i$'s. Clearly the map 
$\varphi$ which maps each $a_i$ to $A_i$ extends uniquely 
to an $\lsc$\--embedding of $L_0$ into $\lalin(X)$ which 
has the required properties.
\smallskip

{\it Case~2}: $d>0$.
The upper semi-lattice $L_0^-$ generated by $a_1,\dots,a_{r-1}$ is an
$\lasc^*$\--substructure of $L_0$ to which the induction hypothesis
applies. This gives for some integer $m$ a special linear set $B\subseteq K^m$
over $K$ and an $\lsc$\--embedding $\psi\colon L_0^-\to\lalin(B)$ having the
required properties. Let $C=\varphi(\UN_{L_0^-}\meet a_r)$ and $n=\scdim a_r$.
Proposition~\ref{pr:treillis-des-lineaires-speciaux} gives a special
linear set $A\subseteq K^{m+n}$ such that $A\cap B=C$. One can extend $\psi$ to an
$\lsc$\--embedding $\varphi$ of $L_0$ into $\lalin(A\cup B)$ exactly like in
the proof of Proposition~\ref{pr:linear-repres}. Then $\varphi$ inherits
from $\psi$ the required properties because all the elements $x\in
L_0$ such that $\asc(x)\neq0$ already belong to $L_0^-$. Indeed, $a_r$ is
the only $\join$\--irreducible element of $L_0$ which doesn't belong to
$L_0^⁻$, so every $a\in L_0\setminus L_0^-$ is greater than $a_r$. But $\scdim
a_r=d>0$ implies that $\scdim a>0$, hence $\asc(a)=0$ by
\assref{ASS:Atk-adb}. Moreover $\varphi(a)\geq\varphi(a_r)=A$ contains
infinitely many atoms (because $\dim A=d>0$), and the conclusion
follows. 
\end{proof}

Let $\sazar(K,d)$, $\salin(K,d)$, $\sadef(K,d)$ denote the class of
all ASC-lattices $L^{\rm At}$ for $L$ ranging over $\szar(K,d)$,
$\slin(K,d)$, $\sdef(K,d)$ respectively. 

\begin{corollary}\label{cor:repres-ASC-non-standard}
  For every integer $d\geq0$, the universal theories of $\sadef(\cK,d)$
  (resp. of $\sazar(K,d)$ or $\salin(K,d)$) is the same for every
  $o$\--minimal or $P$\--minimal expansion of a field $K$ (resp. every
  infinite field $K$). This is the theory of sub-ASC-lattices. 
\end{corollary}

\begin{proof}
Since $\salin(K,d)$ is contained in the other classes, all of which
are contained in the class of ASC-lattices, it suffices to prove that
conversely every sub-ASC-lattice $\lasc$\--embeds into an
ultraproduct of elements of $\salin(K,d)$. By the model-theoretic
compactness theorem, it suffices to prove it for any finitely
generated sub-ASC-lattice $L_0$. 

By Theorem~\ref{th:TCS-treillis-type-fini}, $L_0$ is finite. For any
integer $N\geq0$ let $\varphi_N:L_0\to\lalin(X_N)$ be an $\lsc$\--embedding given
by Proposition~\ref{pr:repres-lin-ASC}. Let $\cU$ be a non principal
ultrafilter in the Boolean algebra of subsets of $\NN$, and consider the
ultraproduct $L=\prod_{N\in\NN}\lalin(X_N)/\cU$. Then $\varphi=\prod_{N\in\NN}\varphi_N/\cU$ is an
$\lsc$\--embedding of $L_0$ into the $L$. In order to prove that it is
an $\lasc$\--embedding, by Remark~\ref{re:lasc-plongement} it remains
check that for every atom $a$ of $L_0$, $\asc(\varphi(a))=\asc(a)$. So let
$a$ be an atom of $L_0$ and $k=\asc(a)$.

If $k>0$ then for every $N\geq k$, $\lalin(X_N)\models\AT_k(\varphi_N(a))$ by
construction. So $L\models\AT_k(\varphi(a))$, that is $\asc(\varphi(a))=k$. 

If $k=0$, let $l$ be any strictly positive integer. For every $N\geq l$,
$\lalin(X_N)\models\AT_N(\varphi_N(a))$ by construction, hence
$\lalin(X_N)\not\models\AT_l(\varphi_N(a))$. So $L\not\models\AT_l(\varphi(a))$, and this being
true for every $l>0$ it follows that $\asc(\varphi(a))=0$. 
\end{proof}

\section{Model-completion of atomic scaled lattices}
\label{se:mod-comp-ASC}

Let us call {\df super ASC\--lattices} those ASC-lattices 
which satisfy the following axioms, all of which are axiomatizable by 
$\forall\exists$\--formulas in $\lasc$. We are going to show that this 
theory is the model-completion of the theory of sub-ASC-lattices 
of dimension at most $d$ (resp. exactly $d$). 
\begin{description}
  \item[\bf Atomicity] 
    Every element $x$ is the least upper bound of the set of 
    atoms smaller than $x$. 
  \item[\bf Catenarity]
    For every non-negative integers $r\leq q\leq p$ and every elements $c\leq
    a\neq\ZERO$, if $c$ is $r$\--sc-pure and $a$ is $p$\--sc-pure then
    there exists a non-zero $q$-sc-pure element $b$ such that $c\leq b\leq a$.
  \item[\bf ASC-Splitting]
    For every $b_1,b_2,a$, if $b_1\join b_2\ll a\neq\ZERO$ and $\pc^0(a)=\ZERO$ 
    there exists non-zero elements $a_1\geq b_1$ and $a_2\geq b_2$ 
    such that:
    \[
      \left\{
        \begin{array}{l}
          a_1=a-a_2\\
          a_2=a-a_1\\
          a_1 \meet a_2 = b_1\meet b_2 
        \end{array}
      \right.
    \]
\end{description}

\begin{remark}\label{re:infinite-d'atomes}%
  An immediate consequence of the atomicity axiom is that for every
  elements $x,y$ in a super ASC-lattice $L$ such that $y<x$ and
  $\scdim(x-y)\geq1$, there are infinitely many atoms $a\in L$ such that $a\leq
  x$ and $a\meet y=\ZERO$. Indeed let $A$ be the set of atoms $a\in L$ such
  that $a\leq x-y$, and $B$ the subset of those $a$ such that $a\meet y=\ZERO$.
  Assume for a contradiction that $B$ is finite and let
  $b=\jjoin_{a\in B}a$. Note that $b\leq y$ and $\scdim b=\dim b=0$. Then by
  the Atomicity axiom
  \[
    x-y=\jjoin_{a\in A}a\leq y\join b,\mbox{ hence } x-y\leq (y\join b)-y=b-y\leq b.
  \]
  This implies that $\scdim(x-y)\leq \scdim b=0$, a contradiction. 
\end{remark}

The notions of ASC\--primitive tuples and ASC\--primitive extensions
are defined for sub-ASC-lattices exactly like for subscaled lattices.
Here is a typical example of what we are going to call an ASC-signature.

\begin{example}\label{ex:ASC-sig}
  Let $L_0$ be a finite sub-ASC-lattice, and $L$ an $\lsc$\--extension
  of $L_0$ generated by a (necessarily unique) SC\--primitive tuple
  $(x_1,x_2)$. Let $(g,\{h_1,h_2\},q)$ be the SC\--signature of $L$ in
  $L_0$ and $k_i=\asc(x_i)$. The following properties are immediate.
\begin{enumerate}
  \item
    If $q<\scdim g$ then $k_1=k_2$ (because $x_1=x_2$ in that case).
  \item
    If $q\neq 0$ then $k_1=k_2=0$ (because each $x_i$ has sc\--pure
    dimension $>0$ in that case) .
  \item
    If $k_1=0$ or $k_2=0$ then $\asc(g)=0$ (because $g\geq x_1\join x_2$).
  \item 
    If $k_1\neq 0$, $k_2\neq 0$ and $\scdim g=0$ then $\asc(g)=k_1+k_2$
    (because $g=x_1\join x_2$ in that case).
\end{enumerate}
\end{example}

We define {\df ASC-signatures} in a finite sub-ASC-lattice $L_0$ as
triples $(g,H,q)$ with $H$ a set of non-necessarily distinct couples
$(h_1,k_1)$, $(h_2,k_2)$ in $L_0\times\NN$, such that $(g,\{h_1,h_2\},q)$ is a
SC\--signature in the $\lsc$\--reduct of $L_0$ and all the conditions
enumerated in Example~\ref{ex:ASC-sig} hold true. In particular we
call the ASC\--signature in this example the {\df ASC\--signature} of
$L$ and of $(x_1,x_2)$ in $L_0$. Note that if $q<\scdim g$ then
$h_1=h_2$ because $(g,\{h_1,h_2\},q)$ is a SC\--signature.

The same argument as in
Proposition~\ref{pr:SC-irr-sign-min}.\ref{it:SC-sgn-ext} shows (using
Remark~\ref{re:lasc-plongement}) that two SC\--primitive extensions of
a finite sub-ASC-lattice $L_0$ are $\lasc$\--isomorphic over $L_0$ if
and only if they have the same ASC\--signature in $L_0$. 

\begin{lemma}\label{le:signature-et-extensions-ASC}
Let $L_0$ be a finite $\lasc$\--substructure of a super ASC-lattice
$\hat L$. Let $\sigma_{\rm At}=(g,q,\{(h_1,k_1),(h_2,k_2)\})$ be an
ASC-signature in $L_0$. Assume that $q\neq0$ or $k_1k_2\neq0$. Otherwise
assume that $\hat L$ is $\aleph_0$\--saturated. Then there exists a
primitive tuple $(x_1,x_2)\in\hat L$ over $L_0$ whose ASC-signature is
$\sigma_{\rm At}$. 
\end{lemma}

\begin{proof}
Let $\sigma=(g,\{h_1,h_2\},q)$. This is a SC\--signature in $L_0$ (more
precisely in its $\lsc$\--reduct).
\smallskip

{\it Case~1:} $\scdim g \geq 1$ and $q\geq1$. Then $\pc^0(g)=0$ and by
definition of ASC\--signatures $k_1=k_2=0$. By
Remark~\ref{re:sign-dans-hat-L} there is an SC\--primitive tuple
$(x_1,x_2)$ in $\hat L$ with signature $\sigma$ in $L_0$. Moreover each
$\asc x_i=0$ (because $\scdim x_i=p\geq 1$) and each $k_i=0$, so the
ASC\--signature of $(x_1,x_2)$ is $\sigma_{\rm At}$.
\smallskip

{\it Case~2:} $\scdim g\geq 1$ and $q=0$. 
Then $\pc^0(g)=0$ again and since $\scdim(h_1\join h_2)<q=0$ by definition
of SC\--signatures we get that $h_1=h_2=\ZERO$. Finally $k_1=k_2$ by
definition of ASC-signatures since $q=0<\scdim g$. 
By Remark~\ref{re:infinite-d'atomes} there are
infinitely many atoms $z$ in $\hat L$ such that $z\leq g$ and $z\meet
g^-=\ZERO$. If $k_1>0$ let $x$ be the join of $k_1$ such atoms of
$\hat L$. Otherwise $\hat L$ is $\aleph_0$\--saturated by assumption hence
it contains an element $x\leq g$ of dimension $0$ such that
$x\meet g^-=\ZERO$ and $\hat L(x)$ has infinitely many atoms.
By the Atomicity Property $\asc(x)=0$. So in both cases $(x,x)$ is 
an SC\--primitive tuple over $L_0$ with ASC\--signature $\sigma_{\rm At}$.
\smallskip

{\it Case~3}: $\scdim g=0$.
Then $q=0$, $g$ is an atom of $L_0$ and $h_1=h_2=\ZERO$. In each of
the two remaining sub-cases, we build a tuple $(x_1,x_2)$ and leave as an
exercise to check that $(x_1,x_2)$ is SC\--primitive over $L_0$ with
ASC\--signature $\sigma_{\rm At}$. 

If $k_1$ and $k_2$ are non-zero then $\asc(g)=k_1+k_2$ 
hence $\hat L(g)$ contains $k_1+k_2$ atoms. Let $x_1$ be the join 
of $k_1$ of them and $x_2$ be the join of the others.

Otherwise, by symmetry we can assume that $k_1=0$. Then $\asc(g)=0$ by
definition of ASC\--signatures so $\hat L(g)$ contains infinitely many
atoms. By $\aleph_0$\--saturation it follows that $\hat L$ contains an
element $x$ smaller than $g$ such that both $\hat L(x)$ and $\hat L(g-x)$
contain infinitely many atoms, hence $\asc(x)=\asc(g-x)=0$. If $k_2=0$
let $(x_1,x_2)=(x,g-x)$. Otherwise let $x_2$ be the join of
$k_2$ atoms in $\hat L(g)$ and let $x_1=g-x_2$. 
\end{proof}

\begin{theorem}\label{th:model-completion-ASC}
The theory of super ASC-lattices of sc\--dimension at most $d$ 
(resp. exactly $d$) is the model-completion of the theory of 
ASC-lattices of dimension at most $d$ (resp. exactly $d$). In
particular, it eliminates the quantifiers in $\lasc$.
It admits $\aleph_0$ completions, each of which is decidable, 
and it is decidable.
\end{theorem}

\begin{proof}
We first only sketch the proof of the first statement, as it
essentially the same as for Theorem~\ref{th:model-completion}. 

On one hand, given a finite sub-ASC-lattice $L_0$, we can embed it in
an extension satisfying the Atomicity and Catenarity Property by
Proposition~\ref{pr:repres-lin-ASC}, and the ASC\--Splitting Property
by means of Lemma~\ref{le:embed-split} applied to any $a,b_1,b_2\in L_0$
such that $b_1\join b_2\ll a\neq\ZERO$ and $\pc^0(a)=\ZERO$ (note that this
last assumption ensures that the extension built in
Lemma~\ref{le:embed-split} is an $\lasc$\--extension). That every
existentially closed sub-ASC-lattice is a super ASC-lattice then
follows, by the model-theoretic compactness theorem. 

On the other hand, given an $\aleph_0$\--saturated super ASC-lattice $\hat
L$, a finite $\lasc$\--substructure $L_0$ and a finite extension $L$
of $L_0$, we reduce to the case where $L$ is SC\--primitive and let
$\sigma$ be its ASC\--signature in $L_0$.
Lemma~\ref{le:signature-et-extensions-ASC} gives an SC\--primitive
extension $L_1$ of $L_0$ in $\hat L$ with the same signature in $L_0$,
hence an embedding of $L_1$ into $\hat L$ over $L_0$ (which maps $L$
to $L_1$). This proves the first statement. 

Quantifier elimination follows, as usual for the model-completion of a
universal theory. Moreover there are finitely many $\emptyset$\--generated
subscaled lattices of dimension at most $d$ (resp. exactly $d$). Each
of them (except the trivial ones, in which $\ZERO=\UN$) can be
enriched with $\aleph_0$ different structures of sub-ASC-lattices obtained
as follows: given a finite $d$\--subscaled lattice $L$ and a
partition\footnote{Necessarily $X_k=\emptyset$ for all but finitely many
  $k$'s, see Footnote~\ref{fn:partition}.} $(X_k)_{k\in\NN}$ of the set of
  atoms $a$ of $L$ such that $\pc^0(a)=a$, we let $\asc(a)=k$ for
  every $a\in X_k$; we then expand $L$ to an $\lasc$\--structure
  according to \assref{ASS:Atk-atomes}. So the completions of the
theory of super ASC\--lattices, which are determined by their prime
model, can be recursively enumerated. 
\end{proof}

We say that a sub-ASC-lattice $L$ is {\df standard} if every element
of sc-dimension $0$ belongs to some $\AT_k(L)$ for some $k>0$. The
existence of standard super ASC-lattices (see
Section~\ref{se:appli-p-adic}) and non-standard super ASC-lattices (by
the model theoretic compactness theorem) implies that the theory of
super ASC-lattices containing a given prime sub-ASC-lattice is not
$\aleph_0$\--categoric, contrary to what happens for super scaled lattice.
However we can recover $\aleph_0$\--categorical by restricting to standard
models.

\begin{proposition}\label{pr:ASC-categoric}
  Let $L_1$, $L_2$ be two {\em standard} countable super ASC-lattices.
  Then every $\lasc$\--isomorphism from a finite sub-ASC-lattice
  $L_{1,0}\subset L_1$ to a sub-ASC-lattice $L_{2,0}\subset L_2$ extends to an
  $\lasc$\--isomorphism from $L_1$ to $L_2$. In particular $L_1$ and
  $L_2$ are isomorphic if and only if their prime
  $\lasc$\--substructures (those generated by the empty set)
  are isomorphic. 
\end{proposition}

\begin{proof}
Let $\varphi$ be an $\lasc$\--isomorphism from $L_{1,0}$ to $L_{2,0}$. Pick
any element $x\in L_1\setminus L_{1,0}$. The subscaled lattice generated in
$L_1$ by $L_{1,0}\cup\{x\}$ (more precisely their $\lsc$\--reducts) is
finite hence by
Proposition~\ref{pr:SC-irr-sign-min}.\ref{it:SC-prim-min} there is a
chain $L_{1,0}\subset L_{1,1}\subset\cdots\subset L_{1,r}$ of SC\--primitive extensions of
subscaled lattices such that $L_{1,0}\cup\{x\}\subseteq L_{1,r}$. Endow each
$L_{1,i}$ with the $\lasc$\--structure induced by $L_1$. It suffices
to prove that $\varphi$ extends to an $\lasc$\--embedding $\varphi_1:L_{1,1}\to
L_2$. Indeed, repeating the argument will give an $\lasc$\--embedding
$\varphi_r:L_{1,r}\to L_2$ extending $\varphi$, and by symmetry the conclusion will
then follow by a back and forth argument.

Identifying $L_{1,0}$ with its image by $\varphi$ we can replace $L_{1,0}$
and $L_{2,0}$ by a common $\lasc$\--structure $L_0$ of $L_1$ and
$L_2$. Now $L_{1,1}$ is generated over $L_0$ by an SC\--primitive tuple
$(x_1,x_2)$ with signature $\sigma_{\rm At}= (g,\{(h_1,k_1),(h_2,k_2)\},q)$. 
In particular $q=\scdim x_i$ and $k_i=\asc(x_i)$ for $i=1,2$.
If $q=0$ then for each $i$, $\scdim x_i=0$ hence $k_i>0$ because $L_1$
is standard. In other words $q\neq0$ or $k_1k_2\neq0$ hence
Lemma~\ref{le:signature-et-extensions-ASC} gives an sc-primitive
tuple $(y_1,y_2)$ in $L_2$ with signature $\sigma_{\rm At}$. Let
$L_{2,1}$ be the $\asc$\--substructure of $L_2$ generated by
$L_{1,1}\cup\{y_1,y_2\}$. By
Proposition~\ref{pr:SC-irr-sign-min}.\ref{it:SC-sgn-ext} $\varphi$ extends
to an $\lsc$\--isomorphism $\varphi_1$ from $L_{1,1}$ to $L_{2,1}$ which
maps each $x_i$ to $y_i$. By construction $\asc(x_i)=\asc(y_i)$, and
by Proposition~\ref{pr:SC-irr-sign-min}.\ref{it:SC-Irr} $\varphi_1$ is the
identity map on $L_0$, so $\asc(\varphi_1(z))=\asc(z)$ for every
$z\in\cI(L_{1,1})$. Hence $\varphi_1$ is an $\lasc$\--isomorphism by
Remark~\ref{re:lasc-plongement}, which proves the result.
\end{proof}

\section{Applications to lattices of $p$-adic semi-algebraic sets}
\label{se:appli-p-adic}

In this section $K$ denotes a fixed $p$\--adically closed field.
For every semi-algebraic set $X$ contained in $K^m$ we let $L(X)$
denote the lattice of semi-algebraic subsets of $X$ closed in $X$,
endowed with its natural structure of ASC-lattice. Note that every
$A\in L(X)$ of dimension $0$ is finite, hence $L(X)$ is {\em standard}. 

As already mentioned in the introduction, the results of the previous
section lead us to conjecture in \cite{darn-2006} and finally to prove
in \cite{darn-2017b} the following result.

\begin{theorem}[Theorem~3.4 in \cite{darn-2017b}]\label{th:triang}
  Let $X$ be a non-empty semi-algebraic subset of $K^m$ without
  isolated points. Assume that $X$ is open in its topological closure
  $\overline{X}$ and let $Y_1,\dots,Y_s$ be a collection of closed
  semi-algebraic subsets of $\partial X=\overline{X}\setminus X$ such that $Y_1\cup\cdots
  \cup Y_s = \partial X$. Then there is a partition of $X$ in non-empty
  semi-algebraic sets $X_1,\dots,X_s$ such that $\partial X_i = Y_i$ for $1 \leq i \leq
  s$.
\end{theorem}

We can now combine this theorem with the results of
Section~\ref{se:appli-p-adic} in order to get the following
applications. 

\begin{theorem}\label{th:Ldef-p-adic-super}
  Let $X$ be any semi-algebraic subset of $K^m$. Then $L(X)$ is a
  super ASC-lattice. In particular its complete theory is decidable
  and eliminates quantifiers in $\lasc$. 
\end{theorem}

\begin{proof}
By construction $L(X)$ is an ASC-lattice satisfying the Atomicity
property. The Catenarity Property will be proved in the
  appendix in much more general settings
  (Proposition~\ref{pr:dp-cat}). We focus here to  the
Splitting Property. So let $A,B_1,B_2\in L(X)$ such that $B_1\cup B_2\ll A$
and $A$ has no isolated point. 

The same holds true for their closures in $K^m$, denoted
$\overline{A}$, $\overline{B}_1$, $\overline{B}_2$. 
Indeed $\overline{A}\setminus A\ll \overline{A}$ and
\begin{displaymath}
  \overline{A}\setminus (\overline{B}_1\cup\overline{B}_2) \subseteq 
  \big(\overline{A}\setminus A\big)\cup\big(A\setminus (B_1\cup B_2)\big).
\end{displaymath}
Apply Theorem~\ref{th:triang} to
$W=\overline{A}\setminus(\overline{B}_1\cup\overline{B}_2)$, $Y_1=\overline{B}_1$
and $Y_2=\overline{B}_2$. It gives a partition of $W$ in
non-empty semi-algebraic sets $W_1$, $W_2$ whose frontiers are respectively
$\overline{B}_1$, $\overline{B}_2$. Then
$\overline{W}_1\cup\overline{W}_2=\overline{A}$,
$\overline{W}_1\cap\overline{W}_2=\overline{B}_1\cap\overline{B}_2$ and each
$\overline{W}_i=W_i\cup\overline{B}_i$. Let $A_1=\overline{W}_1\cap A=(W_1\cap
A)\cup B_1$ and define $A_2$ accordingly. We have to check that
$A_1$, $A_2$ split $A$ along $B_1$, $B_2$. 

$A$ is dense in $\overline{A}=\overline{W}$ and
$W_1=\overline{W}\setminus(W_2\cup\overline{B}_1\cup\overline{B}_2)
=\overline{W}\setminus(\overline{W}_2\cup\overline{B}_1)$ is open in
$\overline{W}$, hence $A\cap W_1$ is dense in $W_1$. In particular $A\cap
W_1\neq\emptyset$, and symmetrically $A\cap W_2\neq\emptyset$. Clearly $A_1\cup A_2=A$, $A_1\cap
A_2=B_1\cap B_2$ and each $A_i\supseteq B_i$ by construction. So it only remains
to check that $A-A_1=A_2$ in $L(X)$, that is that the closure of $A\setminus
A_1$ in $X$ (hence in $A$) is $A_2$. Note that $A\setminus
A_1=A\setminus\overline{W}_1$ and
\[
  \overline{A}\setminus\overline{W}_1
  =\big(W_1\cup W_2\cup\overline{B}_1\cup\overline{B}_2\big)\setminus(W_1\cup\overline{B}_1\big)
  =W_2\cup\big(\overline{B}_2\setminus\overline{B}_1\big).
\]
In particular $A\setminus A_1=A\setminus\overline{W}_1=(\overline{A}\setminus\overline{W}_1)\cap
A$ contains $W_2\cap A$ and is contained in $(W_2\cup\overline{B}_2)\cap
A=\overline{W}_2\cap A=A_2$. The conclusion will follow, if we can prove
that $W_2\cap A$ is dense in $A_2$. Since $A_2=(W_2\cap A)\cup B_2$ it suffices
to check that $B_2\subseteq\overline{W_2\cap A}$. But this is clear since $W_2\cap
A$ is dense in $W_2$, hence in $\overline{W}_2=W_2\cup\overline{B}_2$.
\end{proof}

\begin{corollary}\label{co:appli-Ldef-Km}
  Let $F$ be a $q$\--adically closed field (for some prime $q$ not
  necessarily equal to $p$). Let $X\subseteq K^m$ and $Y\subseteq F^n$ be two
  semi-algebraic sets. 
  \begin{enumerate}
    \item 
      If  $m=n$, $K\preccurlyeq F$ and $X=Y\cap K^n$ then $L(X)\preccurlyeq L(Y)$.
    \item\label{it:appli-equiv}
      $L(X)\equiv L(Y)\iff$ their prime $\lasc$\--substructures are isomorphic. \\
      In particular $L(K^m)\equiv L(F^n)$ if and only if $m=n$. 
    \item\label{it:appli-isom}
      If $K$ and $F$ are countable then
      $L(X)\equiv L(Y)\iff L(X)\simeq L(Y)$.
  \end{enumerate}
\end{corollary}

\begin{proof}
  The two first points follow immediately from
  Theorem~\ref{th:Ldef-p-adic-super}. Note that $L(K^m)\equiv L(F^m)$ is
  a special case because their prime sublattice is just the
  two-element lattice with the same $\lasc$\--structure, because $K^m$
  and $F^m$ both have pure dimension $m$. The last point follows from
  Proposition~\ref{pr:ASC-categoric} since both $L(X)$ and $L(Y)$ are
  standard and countable.
\end{proof}

Given a pair of semi-algebraic sets $X\subseteq K^m$ and $Y\subseteq F^n$, we say that
a homeomorphism $\psi:X\to Y$ is {\df pre-algebraic} if for every
semi-algebraic sets $A\subseteq X$ and $B\subseteq Y$ defined over $K$ and $F$
respectively, $\psi(A)$ and $\psi^{-1}(B)$ are still semi-algebraic sets
defined over $K$ and $F$. It is obviously sufficient to check this for
semi-algebraic sets $A$, $B$ closed in $X$, $Y$ respectively. In
other words, a bijection $\psi:X\to Y$ is a pre-algebraic if and only if
taking direct images by $\psi$ defines an $\lasc$\--isomorphism from
$L(X)$ to $L(Y)$ (which also ensures that $\psi$ is a homeomorphism). 
When $K=F$, semi-algebraic homeomorphisms are obviously pre-algebraic.
The converse is false, as the following example shows. 

\begin{example}\label{ex:pre-vs-semi}
  Assume that the $p$\--valuation of $K$ has value group $\ZZ$, and let
  $R$ be its valuation ring. Applying Theorem~\ref{th:pre-alg-iso}
  below to $X=K$ and $Y=R$ gives a pre-algebraic homeomorphism $\varphi:K\to
  R$. Since its value group is $\ZZ$, the $p$\--valuation defines
  a metric on $K$ and its completion $K'$ is known to be an
  elementary extension of $K$. If $\varphi$ would be semi-algebraic, it
  would then uniquely extend to a semi-algebraic homeomorphism from
  $K'$ to its $p$\--valuation ring $R'$. But this is not possible
  because $R'$ is compact and $K'$ is not. Thus $\varphi$ is not
  semi-algebraic. 
\end{example}

\begin{theorem}\label{th:pre-alg-iso}
  Let $K$, $F$ be {\em countable} $p$\--adically closed fields, and $X\subseteq
  K^m$, $Y\subseteq F^n$ be two semi-algebraic sets. Let $L^0(X)$ and $L^0(Y)$
  be the prime $\lasc$\--substructures of $L(X)$ and $L(Y)$ respectively.
  Then $X$ and $Y$ are pre-algebraically homeomorphic if and only if
  $L^0(X)$ and $L^0(Y)$ are $\lasc$\--isomorphic. In particular, any
  two semi-algebraic sets over $K$ and $F$ with the same pure dimension $d\geq1$
  are pre-algebraically homeomorphic.
\end{theorem}

\begin{proof}
One direction is obvious: every pre-algebraic homeomorphism $\psi:X\to Y$ 
induces an $\lasc$\--isomorphism from $L(X)$ to $L(Y)$, which maps
their respective prime $\lasc$\--substructures one to each other.
Conversely, assume that an $\lasc$\--isomorphism is
given from $L^0(X)$ to $L^0(Y)$. By Proposition~\ref{pr:ASC-categoric}
it extends to an $\lasc$\--isomorphism $\varphi:L(X)\to L(Y)$. For every $t\in
X$, $\varphi$ maps $\{t\}$ to an atom $\{t'\}$ of $L(Y)$. Let $\psi(t)=t'$,
this defines a bijection $\psi:X\to Y$ such that $\psi(A)=\varphi(A)$ for every $A\in
L(X)$, hence $\psi$ is a pre-algebraic homeomorphism. The last statement
follows.
\end{proof}

\section{Appendix: scaled lattices in tame topological structures}
\label{se:appendix}

We have claimed that $\ldef(X)$ in Example~\ref{ex:Ldef} is a scaled
lattice. In order to prove this, we first need a simpler axiomatisation of
scaled lattices.

\begin{fact}\label{fa:pure-dim}
  Let $L$ be a co-Heyting algebra and $a\in L$ an $i$\--pure element.
  For every $b\in L$, if $\dim b<i$ then $a-b=a$. 
\end{fact}

\begin{proof}
  Let $b'=a-b$, and assume for a contradiction that $b'\neq a$. Then
  $\dim a-b'=i$ because $a$ is $i$\--pure. But $a=(a\meet b)\join(a-b)$ by
  \tcref{TC:a=(a-inter-b)-union-(a-b)}, so $a-b'=(a\meet b)-(a-b)$ by
  \tcref{TC:(x-union-y)-z}. In particular $a-b'\leq a\meet b\leq b$, so $\dim
  a-b'\leq \dim b<i$, a contradiction.
\end{proof}

Given a co-Heyting algebra $L$, let us say that an element $a\in L$
has a {\df pure decomposition} in $L$ if for some integer $k$,
$a=\jjoin_{0\leq i\leq k}a_i$ with each $a_i$ an $i$\--pure element of $L$ and
$\dim a_i\meet a_j<\min(i,j)$ for every $i\neq j$. Of course in that case
$\dim a$ is the largest integer $i$ such that $a_i\neq\ZERO$. 

\begin{proposition}\label{pr:unique-pure-dec}
  If an element $a$ in a co-Heyting algebra $L$ has a pure
  decomposition $a=\jjoin_{0\leq i\leq d}a_i$ then $a_d$ is the largest
  $d$\--pure element in $L$ smaller than $a$, and $a-a_d=\jjoin_{0\leq i<
  d}a_i$. In particular, such a pure decomposition (with fixed $d$) is
  unique. 
\end{proposition}

\begin{proof}
  Assume that $b\in L$ is $d$\--pure and $b\leq a$, so that $b-a=\ZERO$.
  For every $i<d$, $\dim a_i< d$ hence $b-a_i=b$ by
  Fact~\ref{fa:pure-dim}. So $b-a=b-a_d$ by
  \tcref{TC:x-(y-union-z)}, hence $b-a_d=\ZERO$ that is $b\leq a_d$. This
  determines $a_d$ as the largest $d$\--pure element in $L$
  smaller than $a$. Moreover $a-a_d=\jjoin_{i<d}(a_i-a_d)$ by
  \tcref{TC:(x-union-y)-z}, and each $a_i-a_d=a_i$ by
  Fact~\ref{fa:pure-dim} (because $a_i$ is $i$\--pure and $\dim a_i\meet
  a_d<i$ by assumption). The uniqueness of the pure decomposition
  follows by decreasing induction. 
\end{proof}

\begin{proposition}\label{pr:axiom-scaled}
  Let $L$ be a $\lsc$\--expansion of a co-Heyting algebra. $L$ is a $d$\--scaled
  lattice if and only if, for every $a\in L$: 
  \begin{list}{$\rm\bf SC_{\theenumi}$:}{\usecounter{enumi}}
  \refstepcounter{enumi}
    \item[$\rm\bf SC_1^d$:]
      \label{SC:pure-dec}
      $a=\jjoin_{0\leq i\leq d}\pc^i(a)$\quad and\quad $\forall i>d$, $\pc^i(a)=\ZERO$. 
    \item
      \label{SC:Ci-i-pur}
      $\forall i,\quad \pc^i(a)$ is $i$\--pure.
    \item 
      \label{SC:dim-inter}
      $\forall i\neq j,\quad \dim\pc^i(a)\meet\pc^j(a)<\min(i,j)$. 
  \end{list}
\end{proposition}

\newcommand{\scun}{$\rm SC_1^d$}

\begin{proof}
  Clearly \ssun{} is \scun. Moreover \scref{SC:scdim-dim} implies that
  \ssref{SS:Ci-inter-Cj}$\Leftrightarrow$\scref{SC:dim-inter} and
  \ssref{SS:Ck(a)-k-pure}$\Leftrightarrow$\scref{SC:Ci-i-pur}. So every
  $d$\--scaled lattice satisfies conditions \scun{} to
  \scref{SC:dim-inter}. Reciprocally, assume that $L$ satisfies these
  conditions. Then it satisfies \scref{SC:scdim-dim} (by
  \scref{SC:Ci-i-pur} and \scun{}) hence also
  \ssun{}, \ssref{SS:Ci-inter-Cj} and \ssref{SS:ll-scdim}.
  The uniqueness of the pure decomposition of $a$ implies that $L$
  satisfies also \ssdeux. 
  
  For every $b\in L$ and every $k\geq \dim_L b$, we have $b=\jjoin_{i\leq
  k}\pc^i(b)$ by \scun{}, and $\pc^k(a)-\pc^i(b)=\pc^k(a)$
  by Fact~\ref{fa:pure-dim} and \scref{SC:Ci-i-pur}). So
  $\pc^k(a)-b=\pc^k(a)-\pc^k(b)$ by \tcref{TC:(x-union-y)-z}, which
  proves \scref{SS:Ck(a)-moins-b=Ck(a)}. 
  
  It remains to check \ssref{SS:Ck(a-union-b)}, for every $a,b\in L$ of
  dimension $\leq k$. Clearly
  $\pc^k(a)\join\pc^k(b)$ is smaller than $a\join b$ and $k$\--pure, hence
  smaller than $\pc^k(a\join b)$ by
  Proposition~\ref{pr:unique-pure-dec}. On the other hand by
  \scun{} and \tcref{TC:(x-union-y)-z}
  \begin{equation}
    (a\join b)-(\pc^k(a)\join\pc^k(b))
    =\jjoin_{i\leq k}(\pc^i(a)\join\pc^i(b))-(\pc^k(a)\join\pc^k(b))
    \leq\jjoin_{i< k}\pc^i(a)\join\pc^i(b).
    \label{eq:finale}
  \end{equation}
  Actually we have equality, by \scun{} and
  Fact~\ref{fa:pure-dim}. Anyway $(a\join b)-(\pc^k(a)\join\pc^k(b))$ has
  dimension $<k$ by (\ref{eq:finale}). On the other hand, by \scun{}
  and \tcref{TC:(x-union-y)-z}, $(a\join b)-(\pc^k(a)\join\pc^k(b))$ is the
  join of $\pc^i(a\join b)-(\pc^k(a)\join\pc^k(b))$ for $i\leq k$. Since 
  $\dim (a\join b)-(\pc^k(a)\join\pc^k(b))<k$ this implies that 
  $\pc^k(a\join b)-(\pc^k(a)\join\pc^k(b))=\ZERO$ hence 
  $\pc^k(a\join b)\leq \pc^k(a)\join\pc^k(b)$. The conclusion follows.
\end{proof}

From now on, let $\cK=(K,\dots)$ be a first-order structure defining a
topology on $K$. Endow $K^m$ with the product topology, and define the
{\df dimension} of a non-empty definable set $X\subseteq K^m$ as
the largest integer $r\geq0$ such that for some coordinate
projection\footnote{A {\df coordinate projection} $\pi:K^m\to K^r$ is a
function defined by $\pi(x_1,\dots,x_m)=(x_{i_1},\dots,x_{i_r})$ for some fixed
$i_1<\cdots<i_r$ in $\{1,\dots,m\}$.} $\pi:K^m\to K^r$, $\pi(X)$ has non-empty
interior. By convention $\dim\emptyset=-\infty$. Recall that for every $x\in X$ the local
dimension $\dim(X,x)$ is the minimum of $\dim U\cap X$ as
$U$ ranges over the definable neighbourhood of $x$. Let $W_k(X)$
denote the set of $x\in X$ such that there is a definable neighbourhood
$B$ of $x$ and a coordinate projection $\pi:K^m\to K^k$ which induces by
restriction a homeomorphism between $B\cap X$ and an open subset of
$K^k$. We say that $\cK$ is a {\df tame topological structure} if it
satisfies the following properties, for every definable sets $X,Y\subseteq K^m$
and every definable function $f:X\to K^n$.

\begin{list}{\bf Dim\theenumi:}{\usecounter{enumi}}
  \item\label{it:dim-image}
    $\dim (f(X))\leq\dim (X)$.
  \item\label{it:dim-union}
    $\dim X\cup Y=\max(\dim X,\dim Y)$.
  \item\label{it:dim-frontier}
    $\dim (X) =\dim(\overline{X})$ and if $X\neq\emptyset$, then $\dim
    (\overline{X}\setminus X)<\dim (X)$.
  \item\label{it:dim-Sk}
    If $\dim (X)=d\geq 0$ then $\dim (X\setminus W_d(X))<d$.
\end{list}

\begin{example}\label{ex:tame-top}
  Ever $o$\--minimal, $C$\--minimal or $P$\--minimal expansion of a
  field $K$ is tame (see \cite{drie-1998}, \cite{hask-macp-1997},
  \cite{darn-cubi-leen-2017}). More generally, every dp\--minimal
  expansion of a field $K$ which is not strongly minimal is tame (see
  \cite{simo-wals-2018}). Following \cite{doli-good-2017-tmp} we may
  also consider the models of visceral theories having finite
  definable choice and no space-filling function: all of them are
  tame. This applies in particular, with the interval topology, to
  every divisible ordered Abelian group whose theory is weakly
  $o$\--minimal.
\end{example}

Note that by (Dim\ref{it:dim-image}), $\dim f(X)=\dim X$ if $f$ is
bijective. For every integer $k\geq 0$ we let 
\begin{displaymath}
  \Delta_k(X)=\big\{x\in X\tq \dim(X,x)=k\big\}.
\end{displaymath}
In particular, $X$ has pure dimensional if and only if $X=\Delta_d(X)$ with
$d=\dim X$. The sets $\Delta_k(X)$ form a partition of
$X$. For every $k\geq0$, $\bigcup_{l\geq k}\Delta_l(X)$ is closed in $X$ (for every
$k$), while $W_k(X)$ is open in $X$.

\begin{proposition}\label{pr:Wk-dense}
  With the above notation and assumptions, $W_k(X)$ is a dense subset
  of $\Delta_k(X)$. If non-empty, they have dimension $k$. In particular,
  $X$ has pure dimension $k$ if and only if $W_k(X)$ is non-empty and
  dense in $X$. 
\end{proposition}

\begin{proof}
  If $x\in W_k(X)$, there is a definable neighborhood $U$ of $x$ in $X$,
  a coordinate projection $\pi:K^m\to K^k$ and an open subset $V$ of $K^k$
  such that $\pi$ induces by restriction a homeomorphism between $U$ and
  $V$. In particular $\dim U=k$ by (dim\ref{it:dim-image}), hence
  $\dim W_k(X)\geq k$. For every sufficiently small neighbourhood $U'$ of
  $x$ in $X$ we have $U'\subseteq U$, hence $\pi$ induces by restriction a
  homeomorphism between $U'$ and an open subset of $K^k$, so $\dim
  U'=k$. This proves that $\dim(X,x)=k$ hence $W_k(X)\subseteq\Delta_k(X)$. 

  We turn now to density. Pick  $x\in\Delta_k(X)$ and a neighbourhood $U$
  of $x$ in $X$. By shrinking $U$ if necessary we may assume that
  $\dim U=k$. From (Dim\ref{it:dim-Sk}) we know that $W_k(U)\neq\emptyset$. On the
  other hand, $W_k(U)\subseteq W_k(X)$ because $U$ is open in $X$.
  Consequently $W_k(U)\subseteq B\cap W_k(X)$ and so $B\cap W_k(X)\neq\emptyset$. This proves
  density.

  By (Dim\ref{it:dim-frontier}) we have $\dim\Delta_k(A)=\dim W_k(A)$, so
  it only remains to check that $W_k(X)$
  has dimension $k$, provided it is not empty. Clearly $\dim W_k(X)\geq
  k$. If $\dim W_k(X)=l>k$ then by (Dim\ref{it:dim-Sk}) $W_l(W_k(X))$
  is non-empty. But $W_k(X)$ is open in $X$, hence $W_l(W_k(X))$ is
  contained in $W_l(X)$. So $W_l(W_k(X))$ is contained both in
  $W_l(X)$ and in $W_k(X)$, a contradiction since $W_l(X)$ and
  $W_k(X)$ are disjoint (they are contained in $\Delta_k(X)$ and $\Delta_l(X)$
  respectively). 

  The last point follows, since $X$ has pure dimension $k$ if and only
  if $X=\Delta_k(X)\neq\emptyset$.
\end{proof}

Recall that $\ldef(X)$ denotes the co-Heyting algebra of all the
definable sets $A\subseteq X$ closed in $X$, expanded by the functions
$(\pc^i)_{i\in\NN}$ defined by $\pc^i(A)=\overline{\Delta_i(A)}\cap A$. 

\begin{proposition}\label{pr:geom-scaled}
  Let $\cK=(K,\dots)$ be a tame topological structure, and
  $X\subseteq K^m$ be a definable set.
  \begin{enumerate}
    \item\label{it:geom-dim}
      For every $A\in L$, $\dim_{\ldef(X)} A=\dim A$. 
    \item\label{it:geom-scaled}
      $\ldef(X)$ is a $d$\--scaled lattice, with $d=\dim X$.
  \end{enumerate}
\end{proposition}

\begin{remark}\label{re:dp-pure-dim}
  The first item ensures that $A\in L$ is $k$\--pure in $\ldef(A)$ if
  and only if it is so in the geometric sense, that is $A=\Delta_k(A)$ or
  equivalently (by Proposition~\ref{pr:Wk-dense}) $W_k(A)$ is dense in
  $A$. 
\end{remark}

\begin{proof}
  In order to ease the notation let $L=\ldef(X)$. 

  (\ref{it:geom-dim}) 
  We can assume that $A\neq\emptyset$. By Fact~\ref{fa:ldim-et-TCdim},
  $\dim_L A$ is then the foundation of rank of $A$ in
  $L\setminus\{\ZERO\}$ for the strong order $\ll$. It suffices to prove,
  by induction on $k$, that $\dim_L A\geq k$ if and only if $\dim
  A\geq k$. This is clear for $k=0$ so let us assume that $k\geq1$ and the
  result is proved for $k-1$. 
  
  If $\dim A\geq k$ there is a coordinate projection $\pi:K^m\to K^k$ and a
  non-empty definable open set $U\subseteq K^k$ contained in $\pi(K)$. Let $Y$
  be any hyperplane of $K^k$ intersecting $U$, and $B=\pi^{-1}(Y\cap U)\cap
  A$. Clearly $Y\subseteq\overline{U\setminus Y}$ hence $B\subseteq\overline{A\setminus B}$, that is
  $B\ll A$. Since $\dim Y=k-1$ we have $\dim B\geq k-1$ by
  (\ref{it:dim-image}), hence $\dim_L B\geq k-1$ by induction hypothesis,
  and finally $\dim_L A\geq k$ since $B\ll A$. 

  Reciprocally, if $\dim_L A\geq k$ by Fact~\ref{fa:ldim-et-TCdim} there
  is $B\in L$ such that $B\ll A$ and $\dim_L B\geq k-1$. By induction
  hypothesis $\dim B\geq k-1$. We have $B\subseteq\overline{A\setminus B}\cap B\subseteq\overline{A\setminus
  B}\setminus(A\setminus B)$, so $\dim B<\dim A\setminus B$ by (Dim\ref{it:dim-frontier}).
  {\it A fortiori} $\dim B< \dim A$ hence $\dim A\geq k$. 
 
  (\ref{it:geom-scaled})
  For every $i\leq m$ and every $A\in\ldef(X)$,
  $\pc^i(A)=\overline{\Delta_i(A)}\cap A=\overline{W_i(A)}\cap A$ by
  Proposition~\ref{pr:Wk-dense}. The scheme \scun{} then follows from
  (Dim\ref{it:dim-Sk}) by a straightforward induction. Moreover each
  $\pc^i(A)$ is $i$\--pure in $L$ by Remark~\ref{re:dp-pure-dim},
  hence \scref{SC:Ci-i-pur} holds true. Finally, for every $i<j$,
  since $W_i(A)$ is open in $A$ and disjoint from $W_j(A)$, it is also
  disjoint from $\overline{W_j(A)}\cap A$ hence
  \begin{displaymath}
    \pc^i(A)\cap\pc^j(A)=\overline{W_i(A)}\cap\overline{W_j(A)}\cap A\subseteq
    \overline{W}_i\setminus W_i.
  \end{displaymath}
  So $\dim\pc^i(A)\cap\pc^j(A)<i$ by (Dim\ref{it:dim-frontier}), which
  proves \scref{SC:dim-inter}. So $\ldef(X)$ is a scaled lattice by
  Proposition~\ref{pr:axiom-scaled}. 
\end{proof}

We turn now to the Catenarity Property. We do not expect it to be
completely general. This property is well known over for $o$\--minimal
fields (it follows immediately from the triangulation theorem). We are
going to prove it for every dp-minimal expansion $\cK=(K,v,\dots)$ of a
non-trivially valued field having definable Skolem functions. This
assumption on Skolem function is somewhat restrictive but it includes
the case of any $p$\--adic field with its semi-algebraic structure (or
even its subanalytic structure), which is sufficient for our needs. We
will use Proposition~3.7 in \cite{simo-wals-2018}, which says that:
    \begin{list}{\bf Dim5:}{}
    \item
      Every definable function $f:X\subseteq K^k\to K^l$ is continuous on a
      definable set $X'$ dense in $X$.
  \end{list}

\begin{proposition}\label{pr:dp-cat}
  Let $\cK=(K,v,\dots)$ be a dp-minimal expansion of a non-trivially
  valued field $(K,v)$ having definable Skolem functions. For every
  non-negative integers $0\leq r<q<p\leq m$ and every definable sets $C\subseteq A\subseteq
  K^m$, if $A$ is $p$\--pure and $\dim C\leq r$, there exists a
  $q$\--pure definable set $B\subseteq A$ such that $C\subseteq\overline{B}$. 
\end{proposition}

The catenarity of $\ldef(X)$, for every definable set $X\subseteq K^m$, follows
immediately. 

\begin{proof}
  We are going to simplify the problem several times, using repeatedly
  the obvious facts that: (i) every open subset of a $p$\--pure set is
  $p$\--pure, and so is its closure; (ii) the union of finitely many
  $p$\--pure sets is $p$\--pure, and; (iii) if $T\subseteq \overline{S}\subseteq K^m$
  then $\pi(T)\subseteq\pi(\overline{S})\overline{\pi(S)}$ for every coordinate
  projection $\pi:K^m\to K^k$. 

  {\it Step 1}. For every $I\subseteq\{1,\dots,m\}$ with $p$ elements let
  $\pi_I:(x_i)_{1\leq i\leq m}\mapsto(x_i)_{i\in I}$ be the corresponding coordinate
  projection. Let $A_I$ be the set of $a\in A$ such that $\pi_I$ induces
  by restriction a homeomorphism between a neighbourhood of $a$ in $A$
  and an open subset of $K^p$, and let $C_I=C\cap\overline{A_I}$. Each
  $A_I$ is $p$\--pure, and by Proposition~\ref{pr:Wk-dense} their
  union is dense in $A$, hence $C$ is the union of the $C_I$'s. So it
  suffices to find for each $I$ a $q$\--pure definable set $B_I\subseteq I$
  such that $C_I\subseteq\overline{B_I}$, and let $B$ be their union. This
  reduces to the case where $A=A_I$ and $C=C_I$  for some $I$. 

  {\it Step 2}. Observe that that $\pi_I(C)$ is contained in the closure
  of $\pi_I(A)$. By the previous step $\pi_I(A)$ is open in $K^p$,
  hence $p$\--pure. Assume that we can find a $q$\--pure subset $Y$ of
  $\pi_I(A)$ whose closure contains $\pi_I(C)$. Let $B=\pi_I^{-1}(Y)\cap A$,
  this is a $p$\--pure subset of $A$ (because the restriction of $\pi_I$
  to $A$ is a local homeomorphism) and $C\subseteq\overline{B}$. So it
  suffices to solve the problem for $\pi_I(A)$ and $\pi_I(C)$. With other
  words, we can assume that $A$ is an open subset of $K^m$, hence
  $p=m$. 
  
  {\it Step 3}. For every $J\subseteq\{1,\dots,m\}$ with $s\leq r$ elements, let $C_J$
  be the set of $c\in C$ such that $\pi_J$ (defined as in the first
  reduction) induces by restriction a homeomorphism between a
  neighbourhood of $c$ in $C$ and an open subset of $K^s$. If we can
  find for each $J$ a $p$\--pure definable set $B_J\subseteq A$ such that
  $C_J$ is contained in the closure of $B_J$, then we are done by
  letting $B$ be the union of the $B_J$'s. This (and a decreasing
  induction on $r$) reduces to the case where $C=C_J$ for some $J$,
  hence $\pi_J$ is a local homeomorphism from $C$ to an open subset of
  $K^r$. Reordering the coordinates if necessary we can then assume
  that $J=\{1,\dots,r\}$. 

  {\it Step 4}. Let $Z=\pi_J(C)$ and $X=\pi_J(A)$. We have
  $Z\subseteq\overline{X}$ hence the dimension of $Z\setminus X$ is $<r$ by
  (Dim\ref{it:dim-frontier}). So is the dimension of $\pi_J^{-1}(Z\setminus X)\cap
  C$ (because $\pi_J$ is now a local homeomorphism). Since $C$ is the
  union of $\pi_J^{-1}(Z\setminus X)\cap C$ and $\pi_J^{-1}(Z\cap X)\cap C$, by a
  straightforward induction on the dimension of $C$ this reduces to
  the case where $C=\pi_J^{-1}(Z\cap X)\cap C$, that is $Z\cap X=Z$, or
  equivalently $\pi_J(C)\subseteq\pi_J(A)$. 
  
  {\it Step 5}. Since $\pi_J$ is a local homeomorphism on $C$, over any
  point $z\in \pi_J(C)$ the fibers $C_z=\pi_J^{-1}(z)\cap C$ are discrete,
  hence finite by Proposition~1.1 in \cite{simo-wals-2018}. The same
  holds true in every elementary extension of $\cK$ so, by the
  model-theoretic compactness theorem, their cardinality must be
  uniformly bounded by some integer $N$. For every $k\leq N$ let $C_k$ be
  the set of $c\in C$ such that the fibers of $\pi_J$ over $\pi_J(c)$ has
  cardinality $k$. This is a finite partition of $C$ in definable set.
  It suffices to solve the problem separately for $A$ and each $C_k$,
  which reduces to the case where $C=C_k$ for some $k$. 

  {\it Step 6}. We can find definable Skolem functions $f_1,\dots,f_k$
  from $\pi_J(C)$ to $K^m$ such that for each $z\in\pi_J(C)$, the fiber
  $\pi_J^{-1}(z)\cap C=\{f_1(z),\dots,f_k(z)\}$. For each $l\leq k$ let
  $C_l=f_l(\pi_J(C))$. This is again a finite partition of $C$ in
  definable sets. So the problem boils down to the case where $C=C_l$
  for some $l$, that is $\pi_J$ induces a bijection from $C$ to
  $Z=\pi_J(C)$, and $f=f_1$ is the reciprocal bijection. After this
  reduction we cannot assume anymore that $Z=\pi_J(C)$ is open in $K^r$.
  However, the complement in $Z$ of the interior of $Z$ in $K^r$ has
  dimension $<r$ by (Dim\ref{it:dim-Sk}). By (Dim5) the set of
  discontinuities of $f$ also has dimension $<r$. Hence, by a
  straightforward induction on the dimension of $C$, we can reduce to
  the case where $Z$ is open, and $f:Z\to C$ and the restriction of
  $\pi_J$ are reciprocal homeomorphisms. 

  {\it Step 7}. One can easily check that $C$ is contained in the
  closure of $A'=\pi_J^{-1}(Z)\cap A$. The latter is open. This reduces to
  the case where $A=A'$, that is $\pi_J(A)=\pi_J(C)$. In particular, the
  restriction $\rho$ of $f\circ\pi_J$ to $A\cup C$ then defines a continuous
  retraction onto $C$ (that is $\rho$ is continuous on $A\cup C$ and
  $\rho(c)=c$ for every $c\in C$). 
  \smallskip

  Using this retraction we can now finish the proof. We do it when
  $q=m-1$, that is $q=p-1$, the result for smaller values of $q$
  following immediately by decreasing induction. 
  For every $k\in\{1,\dots,m\}$ and every $a=(a_1,\dots,a_m)\in A$ let
  $\rho_k(a)$ be the $k$\--th coordinate of $\rho(a)$, so that
  $\rho(a)=(\rho_1(a),\dots,\rho_m(a))$. Note that $\pi_J(a)=\pi_J(\rho(a))$ by
  construction, hence $\rho(a)=(a_1,\dots,a_r,\rho_{r+1}(a),\dots,\rho_m(a))$. For each
  $k\in\{r+1,\dots,m\}$ let 
  \begin{displaymath}
    A_k=\big\{a\in A\tq v\big(a_k-\rho_k(a)\big)\geq \min_{l\neq
    k}v\big(a_l-\rho_l(a)\big)\big\}.
  \end{displaymath}
  This is the set of points $a\in A$ such that $a_k$ is not strictly
  further from $\rho_k(a)$ than is $a$ from $\rho(a)$ (see
  Figure~\ref{fi:A-X1-X2}).
  Clearly $A_k$ is definable, open, and $A$ is the union of the
  $A_k$'s. In particular $C$ is contained in the union of the
  $\overline{A_k}$'s. 

  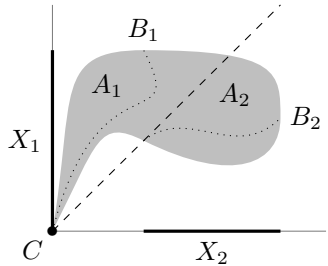
\begin{figure}[h]
    \begin{center}
        \begin{tikzpicture}[scale=.6]
          % A en gris clair + B1 et B2 en pointill\'es
          \begin{scope}
            \clip (0,0) .. controls +(1,2) and +(-1,.5) .. (2,2)
              .. controls +(2,-1) and +(0,-1) .. (5,2.5)
              .. controls +(0,1) and +(3,0) .. (2,4)
              .. controls +(-2,0) and +(.5,3) .. (0,0)
              -- cycle;
            \fill[color=lightgray] (0,0) rectangle (5,4);
            \draw[dotted] (0,0) .. controls +(1,4) and +(1,-2) .. (2,4);
            \draw[dotted] (2,2) .. controls +(1,.8) and +(-1,-1) .. (5,2.5);
          \end{scope}
          % Axes + diagonale + c
          \draw[gray] (6,0) -- (0,0) -- (0,5);
          \draw[dashed] (0,0) node{\small$\bullet$} node[below left]{$C$}  --(5,5);
          % X1, X2 et \'etiquettes
          \draw (4,3) node{$A_2$} 
            (1.2,3.2) node{$A_1$} 
                (2,4) node[above]{$B_1$}
              (5,2.5) node[right]{$B_2$};
          \draw[very thick] (0,0) -- (0,4) node[midway,left]{$X_1$};
          \draw[very thick] (2,0) -- (5,0) node[midway,below]{$X_2$};
        \end{tikzpicture}
      \caption{In $K^2$, the dashed line splits $A$ (in gray) in two
      parts $A_1$, $A_2$.}
      \label{fi:A-X1-X2}
    \end{center}
  \end{figure}
  
  For each $k\in\{r+1,\dots,m\}$ let $\pi_k:K^m\to K^{m-1}$ be the projection
  which forgets the $k$\--th coordinate. Let $\theta_k$ be a definable
  section of the restriction of $\pi_k$ to $A_k$ (given by definable
  Skolem functions). By (Dim5) there is a definable set $X_k\subseteq
  \pi_k(A_k)$ dense in $\pi_k(A_k)$ such that $\theta_k$ is continuous on $X_k$.
  Finally let $B_k=\theta_k(X_k)$ (the dotted lines in
  Figure~\ref{fi:A-X1-X2}). Recall that $A_k$ is open in $K^m$, hence
  so is $\pi_k(A_k)$ in $K^{m-1}$. In particular $\pi_k(A_k)$ is
  $(m-1)$\--pure, hence so is $X_k$. By construction, the restriction
  of $\theta_k$ to $X_k$ is a homeomorphism, so $B_k$ is $(m-1)$\--pure.
  Letting $B$ be the union of the $B_k$'s, it only remains to check
  that $C\subseteq\overline{B}$. 

  In order to do so, pick any $c=(c_1,\dots,c_m)\in C$. There is $k\in\{r+1,\dots,m\}$ such that
  $c\in\overline{A_k}$, hence $\pi_k(c)\in\overline{X_k}$. It suffices to
  prove that $\theta_k(x)$ tends to $c$ as $x$ tends to $\pi_k(c)$ in
  $X_k$, in order to conclude that
  $c\in\overline{B_k}$, and finally that $C\subseteq\overline{B}$. Let
  $\pi_{J,k}:K^{m-1}\to K^r$ be such that $\pi_J=\pi_{J,k}\circ\pi_k$. For every
  $x\in X_k$, let $a=(a_1,\dots,a_m)=\theta_k(x)$ and observe that $\pi_J(a)=\pi_{J,k}(x)$, so 
  \begin{equation}
    \underbrace{f\circ\pi_J(a)}_{=\rho(a)}=f\circ\pi_{J,k}(x)
    \xrightarrow[x\to\pi_k(c)]{}
    f\circ\pi_{J,k}\big(\pi_k(c)\big)=\underbrace{f\circ\pi_J(c)}_{=\rho(c)=c}.
    \label{eq:rho-a-c}
  \end{equation}
  Consequently $\pi_k(\rho(a))\xrightarrow[x\to\pi_k(c)]{}\pi_k(c)$, so
  \begin{displaymath}
    \pi_k(a)-\pi_k\big(\rho(a)\big)=x-\pi_k\big(\rho(a)\big)
    \xrightarrow[x\to\pi_k(c)]{}\pi_k(c)-\pi_k(c)=(0,\dots,0),
  \end{displaymath}
  that is
  \begin{equation}
    \min_{l\neq k}v\big(a_l-\rho_l(a)\big)\xrightarrow[x\to\pi_k(c)]{}+\infty.
    \label{eq:ak-rhok}
  \end{equation}
  We have $a=\theta_k(x)\in A_k$ so, by definition of $A_k$,
  \begin{equation}
    \min_{1\leq l\leq m}v\big(a_l-\rho_l(a)\big)=\min_{l\neq k}v\big(a_l-\rho_l(a)\big)
    \label{eq:lim-Ak}
  \end{equation}
  By(\ref{eq:ak-rhok}) and (\ref{eq:lim-Ak}), we get that $a-\rho(a)$
  tends to $(0,\dots,0)$ as $x$ tends to $\pi_k(c)$. So by (\ref{eq:rho-a-c}) 
  \begin{displaymath}
    \theta_k(x)=a=\rho(a)+\big(a-\rho(a)\big) \xrightarrow[x\to\pi_k(c)]{} c.
  \end{displaymath}
\end{proof}

% \bibliographystyle{plain}
% \bibliography{biblio}

\end{document}